\documentclass[11pt]{article}

\usepackage{amsmath, amsthm,enumerate,cite,color}
\usepackage{etoolbox}
\providetoggle{long}
\usepackage[none]{hyphenat}
\settoggle{long}{true}
\usepackage{amssymb,graphicx}
\usepackage[numbers]{natbib}
\usepackage{cleveref}[2011/12/24]
\usepackage[plain]{fullpage}
\usepackage[text={17cm,24cm}]{geometry}
\usepackage{etoolbox}
\newcommand{\NP}{\ensuremath{\mathsf{NP}}}

\newtheorem{theorem}{Theorem}[section]
\newtheorem{definition}[theorem]{Definition}
\newtheorem{corollary}[theorem]{Corollary}

\newtheorem{observation}[theorem]{Observation}
\newtheorem{claim}{Claim}

\newtheorem*{thm-bigraph}{\Cref{{thm:decomposition-2P3-free-right-Sperner-bigraphs}}}
\newtheorem*{thm-cobigraph}{\Cref{{thm:decomposition-co-2P3-free-right-Sperner-cobigraphs}}}

\newtheorem{lemma}[theorem]{Lemma}
\newtheorem{remark}{Remark}

\newcommand{\bbone}{{\boldsymbol{1}}}
\newcommand{\bbzero}{{\boldsymbol{0}}}
\DeclareMathOperator{\cw}{cw}

\title{Characterizing and decomposing classes of threshold, \\split, and bipartite graphs via $1$-Sperner hypergraphs}

\author{Endre Boros\\
\small MSIS Department and RUTCOR, Rutgers University, New Jersey, USA\\
\small 100 Rockafeller Rd, Piscataway NJ 08854, USA\\
\small \texttt{Endre.Boros@rutgers.edu}\\
\and
Vladimir Gurvich\\
\small National Research University: Higher School of Economics, Moscow, Russia\\
\small \texttt{vgurvich@hse.ru}\\
\and
Martin Milani\v c\\
\small University of Primorska, UP IAM, Muzejski trg 2, SI6000 Koper, Slovenia\\
\small University of Primorska, UP FAMNIT, Glagolja\v ska 8, SI6000 Koper, Slovenia\\
\small \texttt{martin.milanic@upr.si}
}

\date{\today}

\begin{document}

\maketitle

\begin{sloppypar}
\begin{abstract}
A hypergraph is said to be \emph{$1$-Sperner} if for every two hyperedges the smallest of their two set differences is of size one. We present several applications of $1$-Sperner hypergraphs and their structure to graphs. In particular, we consider the classical characterizations of threshold and domishold graphs and use them to obtain further characterizations of these classes in terms of $1$-Spernerness, thresholdness, and $2$-asummability
of their vertex cover, clique, dominating set, and closed neighborhood hypergraphs. Furthermore, we apply a decomposition property of $1$-Sperner hypergraphs to derive decomposition theorems for two classes of split graphs, a class of bipartite graphs, and a class of cobipartite graphs. These decomposition theorems are based on certain matrix partitions of the corresponding graphs, giving rise to new classes of graphs of bounded clique-width and new polynomially solvable cases of several domination problems.
\end{abstract}
\end{sloppypar}


\section{Introduction}\label{sec:introduction}

Hypergraphs are one of the most fundamental and general combinatorial objects, encompassing various important structures such as graphs, matroids, and combinatorial designs. Results for more specific discrete structures (e.g., graphs) can often be proved more generally, in the context of suitable classes of hypergraphs~(see, e.g., \citet{MR1956926}). Of course, applications of hypergraph theory to structural and optimization aspects of other discrete structures are not restricted to the above phenomenon. Since it is impossible to survey here all kinds of applications of hypergraphs, let us mention only a few more. First, recent work of \citet{ACM-TALG} made use of connections between hypergraphs and binary matrices, acyclic digraphs, and partially ordered sets to study two combinatorial optimization problems motivated by computational biology. Second, a common approach of applying hypergraph theory to graphs is to define and study a hypergraph derived in an appropriate way from a given graph, depending on what type of property of a graph or its vertex or edge subsets one is interested in. This includes matching hypergraphs~\cite{MR2852510,MR818499}, various clique~\cite{MR2755907,MR818499,MR1689294,MR539710,MR719998}, independent set~\cite{MR2755907,MR818499}, neighborhood~\cite{MR3314932,MR3281177}, separator~\cite{MR1936236,CM-ISAIM2014}, and dominating set hypergraphs~\cite{MR3281177,CM-ISAIM2014}, etc.
Close interrelations between hypergraphs and monotone Boolean functions can be useful in such studies, allowing for the transfer and applications of results from the theory of Boolean functions (see, e.g., \cite{MR2742439}).

In this work, we present several new applications of hypergraphs to graphs. Our starting point is~\cite{BGM-decomposing-1-Sperner} where the class of $1$-Sperner hypergraphs was studied. It was shown that such hypergraphs can be decomposed in a particular way and that they form a subclass of the class of threshold hypergraphs studied by~\citet{MR562306}, by~\citet{MR791660} and, in the equivalent context of threshold monotone Boolean functions, also by~\citet{MR0439441} and by~\citet{MR798011}. (Precise definitions of all relevant concepts will be given in \Cref{sec:preliminaries}.)

\begin{sloppypar}
While threshold hypergraphs are defined using the existence of certain weights, the $1$-Sperner property is a `purely combinatorial' notion that forms a sufficient condition for thresholdness. Purely combinatorial necessary conditions for thresholdness are also known and follow from the characterization of threshold Boolean functions due to~\citet{Chow} and~\citet{5397278}. The condition states that a hypergraph is threshold if and only if a certain obvious necessary combinatorial condition, called $k$-asummability, for all $k\ge 2$, is satisfied. For certain families of hypergraphs derived from graphs $2$-asummability implies $1$-Spernerness, which means that in such cases all the three (generally properly nested) properties of $1$-Spernerness, thresholdness, and $2$-asummability coincide. Early applications of this idea are implicit in the works of \citet{MR0479384} and~\citet{MR0491342} characterizing threshold and domishold graphs, respectively (see~\Cref{sec:old-results} for details).
By definition, these are graphs that admit a non-negative linear vertex weight function separating the characteristic vectors of all independent sets, resp.~dominating sets~from the characteristic vectors of all other sets. The corresponding hypergraphs derived from a given graph, obtained
using the notion of transversal hypergraphs, are the hypergraphs of inclusion-minimal vertex covers and inclusion-minimal closed neighborhoods, respectively.
\end{sloppypar}

\begin{sloppypar}
More recent examples involve the works of~\citet{MR3281177,ISAIM2014} who made use of a slightly more general class than $1$-Sperner hypergraphs, called dually Sperner hypergraphs, to characterize two classes of graphs defined by the following properties: every induced subgraph has a non-negative linear vertex weight function separating the characteristic vectors of all total dominating sets~\cite{ISAIM2014}, resp.~connected dominating sets~\cite{MR3281177}, from the characteristic vectors of all other sets. In these cases, the corresponding transversal hypergraphs derived from a graph $G$ are the hypergraphs of inclusion-minimal (open) neighborhoods and inclusion-minimal cutsets, respectively.
Due to the close relation between $1$-Sperner and dually Sperner hypergraphs (see~\cite{BGM-decomposing-1-Sperner}), all the results from~\cite{MR3281177,ISAIM2014} can be equivalently stated using $1$-Sperner hypergraphs. In particular, the results of the extended abstract~\cite{ISAIM2014} are stated using the $1$-Sperner property in the full version of the paper~\cite{CM-ISAIM2014}.
\end{sloppypar}

A common approach in these studies is to consider various hypergraphs associated to a given graph and investigate the class of graphs for which the associated hypergraph is threshold. The goal is then to characterize these graph classes and, in particular, to understand under what conditions they are hereditary (that is, it is closed under vertex deletion). In some of the above cases (namely in the cases of threshold and domishold graphs) the resulting graph class is hereditary, while in the other two cases mentioned above the condition for induced subgraphs has to be imposed on the graph as part of the definition, since in general those properties are not hereditary.

\subsection*{Our results}

We present several novel applications of the notion of $1$-Sperner hypergraphs and their structure to graphs. Our results can be summarized as follows.

\medskip
\begin{sloppypar}
\noindent{\bf 1.~Characterizations of threshold and domishold graphs.} We show that forbidden induced subgraph characterizations of threshold and domishold graphs due to \citet{MR0479384} and~\citet{MR0491342} imply that
for several families of hypergraphs derived from graphs, some or all of the three (generally properly nested) properties of $1$-Spernerness, thresholdness, and $2$-asummability coincide.
These families include the vertex cover, clique, closed neighborhood, and dominating set hypergraphs (see Theorems~\ref{thm:threshold-graphs-2},~\ref{thm:threshold},~\ref{thm:domishold}, and~\ref{thm:domishold-2}). For example, we show that in the class of clique hypergraphs of graphs, $1$-Spernerness, thresholdness,
and $2$-asummability are all equivalent and that
they characterize the threshold graphs. Furthermore, threshold graphs are exactly the co-occurrence graphs of 1-Sperner hypergraphs (Theorem~\ref{thm:threshold-graphs-new}).
\end{sloppypar}

\begin{sloppypar}
These results are interesting for multiple reasons. First, they give new characterizations of threshold graphs, which further extend the long list~\cite{MR1417258}. Second, our characterizations of clique hypergraphs of threshold graphs parallels the characterization of clique hypergraphs of chordal graphs due to~\citet{MR719998}. Third, our characterizations of threshold and domishold graphs implicitly address the question by Herranz~\cite{MR2794309} asking about further families of Boolean functions
that enjoy the property that thresholdness and $2$-asummability coincide.
Our results imply that this property is enjoyed by monotone Boolean functions corresponding to various hypergraphs naturally associated to graphs: vertex cover, clique, closed neighborhood, and dominating set hypergraphs. Finally, it is not at all a priori obvious whether the class of graphs whose clique hypergraphs are threshold is hereditary or not. Interestingly, it turns out that it is. Considering instead of the family of maximal cliques the families of edges (or, equivalently, the transversal family of inclusion-minimal vertex covers)
or of inclusion-minimal dominating sets (or, equivalently, the transversal family of inclusion-minimal closed neighborhoods) also results in hereditary graphs classes. In contrast, the inclusion-minimal (open) neighborhoods or inclusion-minimal cutsets do not define hereditary graph classes.
\end{sloppypar}

\medskip

\begin{sloppypar}
\noindent{\bf 2.~Decomposition results for four classes of graphs.}
We apply the decomposition theorem for $1$-Sperner hypergraphs to derive decomposition theorems for four interrelated classes of graphs:
two subclasses of split graphs, a class of bipartite graphs, and a class of cobipartite graphs (Theorems~\ref{thm:decomposition-clique-split-H-free},~\ref{thm:decomposition-independent-split-H-bar-free},
~\ref{thm:decomposition-2P3-free-right-Sperner-bigraphs}, and~\ref{thm:decomposition-co-2P3-free-right-Sperner-cobigraphs}, respectively). This is done by exploiting a variety of ways of how the incidence relation between the vertices and the hyperedges of a hypergraph can be represented with a graph, expressing the $1$-Sperner property of hypergraphs in terms of the corresponding derived graphs, and finally translating the decomposition of $1$-Sperner hypergraphs to decompositions of the corresponding graphs. These decompositions are recursive and are described in terms of matrix partitions of graphs, 
a concept introduced by Feder et al.~\cite{MR2002173} and studied in a series of papers~(see, e.g.,~\cite{MR3090507} and references therein).
\end{sloppypar}

\medskip
\noindent{\bf 3.~New classes of graphs of bounded clique-width.}
We explore several consequences of the obtained graph decomposition results. The first ones are related to clique-width, a graph parameter introduced by~\citet{MR1217156}, whose importance is mainly due to the fact that many \NP-hard decision and optimization problems are polynomially solvable in classes of graphs of uniformly bounded clique-width~\cite{MR2323400,MR1739644,MR2232389}. This makes it important to identify graph classes of bounded, resp.~unbounded clique-width~\cite{MR2536473}. Each of the four considered graph classes obtained from $1$-Sperner hypergraphs is defined by forbidding one particular $6$-vertex graph as induced subgraph (one of the two graphs in Fig.~\ref{fig:small-graphs-1} or their complements)
and an extra condition on the structure of the neighborhoods.

\begin{figure}[h!]
  \centering
   \includegraphics[width=70mm]{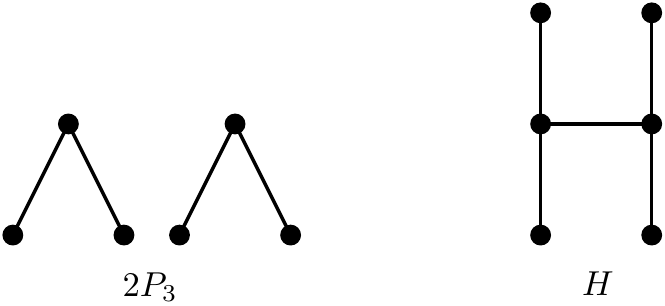}
\caption{The graphs $2P_3$ and $H$.}
\label{fig:small-graphs-1}
\end{figure}

As a consequence of the decomposition theorems, we show that a graph in any of the four classes has clique-width at most $5$ (Theorems~\ref{thm:cwd-clique-split},~\ref{thm:cwd-independent-split},~\ref{thm:cwd-bigraph}, and~\ref{thm:cwd-cobigraph}). This result is in sharp
contrast with the fact that the four graph classes defined by the same forbidden induced subgraphs but without any restrictions on the neighborhoods are of unbounded clique-width. In particular, this is the case for the class of $H$-free split graphs, where $H$ is one of the graphs depicted in Fig.~\ref{fig:small-graphs-1}.

Similar results, but incomparable to ours, were obtained for a class of bipartite graphs defined by two $7$-vertex forbidden induced subgraphs~\citet{MR1755401}; see also~\cite{MR3679842,MR2158639,MR2046648} for related work. Our results for bipartite graphs directly address the remark from~\cite{MR2158639}: {\it ``It is challenging to find new classes of bipartite graphs, defined both by their decomposition tree and another property (\ldots) and to use this for new efficient algorithms design.''} Naturally, our approach is applicable also to split and cobipartite graphs.

\medskip
\noindent{\bf 4.~New polynomially solvable cases of variants of domination.}
Finally, we make further use of the bounded clique-width result by identifying new polynomially solvable cases of three basic variants of the dominating set problem. We show that the dominating set problem, the total dominating set problem, and the connected dominating set problem are all polynomial-time solvable in the class of $H$-free split graphs (Theorem~\ref{thm:dom-sets}). This result is sharp in the sense that all three problems are \NP-hard in the class of split graphs~\cite{LP83,MR761623,MR754426}, as well as in the class of $H$-free graphs. The result for $H$-free graphs follows from the fact that the class of $H$-free graphs is a superclass of the class of line graphs, in which all these problems are known to be \NP-hard~\cite{MR579424,McRae,MR2874128}. Moreover, strong inapproximability results are known for these variants of domination
in the class of split graphs: for every $\epsilon >0$, the dominating set, the total dominating set, and the connected dominating set problem cannot be approximated to within a factor on  \hbox{$(1-\epsilon)\ln n$} in the class of $n$-vertex split graphs, unless ${\sf P} = \NP$. Theorem 6.5 in~\cite{MR3732593} provides this result for domination and total domination, while for connected domination the same result follows from the fact that all the three problems are equivalent in any class of split graphs (see Lemma~\ref{lem:reductions}).

\subsection*{Structure of the paper}

\begin{sloppypar}
In \Cref{sec:preliminaries}, we collect the necessary preliminaries on graphs and hypergraphs, including a decomposition theorem for $1$-Sperner hypergraphs. In~\Cref{sec:old-results}, we consider the characterizations of threshold and domishold graphs due to~\citet{MR0479384} and~\citet{MR0491342} and use them to obtain further characterizations of these classes in terms of $1$-Spernerness, thresholdness, and $2$-asummability properties of their vertex cover, clique, dominating set, and closed neighborhood hypergraphs.
In \Cref{sec:to-graphs}, we associate various types of incidence graphs to a hypergraph and translate the $1$-Spernerness property of the hypregraph to the incidence graphs. In \Cref{sec:structure-graphs}, we apply the decomposition theorem for $1$-Sperner hypergraphs to derive decomposition theorems for four classes of graphs. These theorems are then used in~\Cref{sec:clique-width} to infer that
each of the graph in resulting four classes has clique-width at most $5$.
Finally, in \Cref{sec:domination} we study the consequences of our structural approach for domination problems in a subclass of split graphs.
\end{sloppypar}

\section{Preliminaries}\label{sec:preliminaries}

\subsection{Preliminaries on graphs}

All graphs in this paper will be finite, simple, and undirected.
Let $G = (V,E)$ be a graph. Given a vertex $v\in V$, its \emph{neighborhood} is the set, denoted by $N(v)$, of vertices adjacent to $v$, and its  \emph{closed neighborhood} is the set, denoted by $N[v]$, of
vertices adjacent or equal to $v$. 
An \emph{independent set} (also called a \emph{stable set}) in $G$ is a set of pairwise non-adjacent vertices, and a \emph{clique} is a set of pairwise adjacent vertices. A \emph{dominating set} in $G$ is a set $D\subseteq V$ such that every vertex of $G$ is either in $D$ or has a neighbor in $D$. A \emph{total dominating set} in $G$ is a set $D\subseteq V$ such that every vertex of $G$ has a neighbor in $D$.
A \emph{connected dominating set} in $G$ is a dominating set that induces a connected subgraph. {\sc Dominating Set}, {\sc Total Dominating Set}, and {\sc Connected Dominating Set} are the problems of finding a minimum dominating set, a total dominating set, resp., connected dominating set in a given graph. (Note that a graph has a total dominating set if and only if it has no isolated vertices and it has a connected dominating set if and only if it is connected.) We denote the minimum size of a dominating set, total dominating set, and connected dominating set in a graph $G$ by $\gamma(G)$, $\gamma_t(G)$, and $\gamma_c(G)$, respectively.

Given a set $\mathcal{F}$ of graphs and a graph $G$, we say that $G$ is \emph{$\mathcal{F}$-free} if and only if $G$ has no induced subgraph isomorphic to a member of $\mathcal{F}$. Moreover, for a graph $F$, we say that $G$ is \emph{$F$-free} if it is $\{F\}$-free. Given a graph $G$ and a vertex $u\in V(G)$, we denote by $G-u$ the subgraph of $G$ induced by $V(G)\setminus \{u\}$. We denote the disjoint union of two graphs $G_1$ and $G_2$ by $G_1+G_2$; in particular, we write $2G$ for the disjoint union of two copies of $G$. As usual, we denote by $P_n$, $C_n$, and $K_n$ the path, the cycle, and the complete graph with $n$ vertices, respectively.

\medskip
\noindent{\bf Split, bipartite, and cobipartite graphs.}
A graph is \emph{split} if its vertex set can be partitioned into a (possibly empty) clique and a (possibly empty) independent set. A \emph{split partition} of a split graph $G$ is a pair $(K,I)$ such that $K$ is a clique, $I$ is an independent set, $K\cap I = \emptyset$, and $K\cup I = V(G)$.
A graph is \emph{bipartite} (or, shortly, a \emph{bigraph}) if its vertex set can be partitioned into two (possibly empty) independent sets. A \emph{bipartition} of a bipartite graph $G$ is a pair $(A,B)$ of two independent sets in $G$ such that $A\cap B = \emptyset$ and $A\cup B = V(G)$.
A graph $G$ is \emph{cobipartite} if its complement $\overline{G}$,
that is, the graph with vertex set $V(G)$ in which two distinct vertices are adjacent if and only if they are non-adjacent in $G$, is bipartite.

\medskip
\noindent{\bf Threshold graphs.} A graph $G = (V,E)$ is said to be \emph{threshold} if there exist a non-negative integer weight function $w:V\to \mathbb{Z}_{\ge 0}$ and a non-negative integer threshold $t\in \mathbb{Z}_{\ge 0}$ such that for every subset $X\subseteq V$, we have
$w(X):= \sum_{x\in X}w(x)\le t$ if and only if $X$ is an independent set.
Threshold graphs were introduced by Chv\'atal and Hammer in 1970s~\cite{MR0479384} and were afterwards studied in numerous papers. Many results on threshold graphs are summarized in the monograph by Mahadev and Peled~\cite{MR1417258}.

We now recall three of the many characterizations of threshold graphs. By $P_4$ we denote the $4$-vertex path, by $C_4$ the $4$-vertex cycle, and by $2K_2$ the disjoint union of two copies of $K_2$ (the complete graph of order two).

\begin{theorem}[Chv\'atal and Hammer~\cite{MR0479384}]\label{thm:char1}
A graph $G$ is threshold if and only if $G$ is $\{P_4,C_4,2K_2\}$-free.
\end{theorem}

\begin{theorem}[Chv\'atal and Hammer~\cite{MR0479384}]\label{thm:char2}
A graph $G$ is threshold if and only if $G$ is split with a split partition $(K,I)$ such that there exists an ordering $v_1,\ldots, v_k$ of $I$ such that
$N(v_i)\subseteq N(v_j)$ for all $1\le i<j\le k$.
\end{theorem}

The \emph{join} of two vertex-disjoint graphs $G_1$ and $G_2$ is the graph, denoted by $G_1\ast G_2$, obtained from the disjoint union $G_1+G_2$ by adding to it all edges with one endpoint in $G_1$ and one endpoint in $G_2$.

\begin{sloppypar}
\begin{corollary}\label{cor:threshold}
A graph $G$ is threshold if and only if $G$ can be built from the $1$-vertex graph by an iterative application of the operations of adding an isolated vertex (that is, disjoint union with $K_1$) or a universal vertex (that is, join with $K_1$).
\end{corollary}
\end{sloppypar}

\medskip
\noindent{\bf Domishold graphs.} A graph $G = (V,E)$ is said to be \emph{domishold} if there exist a non-negative integer weight function $w:V\to \mathbb{Z}_{\ge 0}$ and a non-negative integer threshold $t\in \mathbb{Z}_{\ge 0}$ such that for every subset $X\subseteq V$, we have
$w(X)\ge t$ if and only if $X$ is a dominating set.
Domishold graphs were introduced by~\citet{MR0491342}, who characterized them with the set of forbidden induced subgraphs depicted in Fig.~\ref{fig:domishold}.

\begin{sloppypar}
\begin{theorem}[\citet{MR0491342}]\label{thm:char3}
A graph $G$ is domishold if and only if $G$ is $\{P_4,2K_2,K_{3,3},K_{3,3}^+,\overline{2P_3}\}$-free.
\end{theorem}
\end{sloppypar}

\begin{figure}[h!]
  \centering
   \includegraphics[width=130mm]{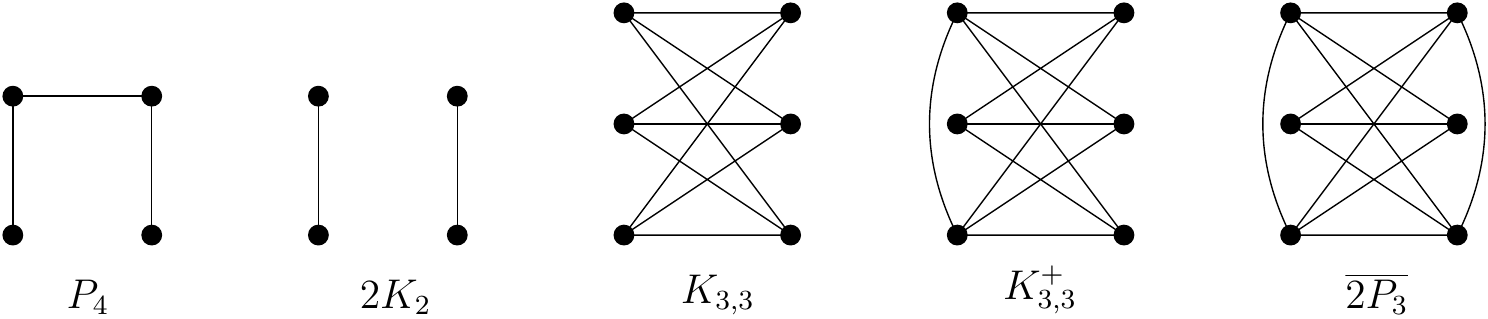}
\caption{Forbidden induced subgraphs for the class of domishold graphs.}
\label{fig:domishold}
\end{figure}

%

\medskip
\noindent{\bf Matrix partitions.} Let $M$ be a symmetric $m\times m$ matrix with entries in $\{0,1,\ast\}$. An \emph{$M$-partition} of a graph $G$ is a partition of the vertex set of $G$ into parts $V_1,\ldots, V_m$ such that two distinct vertices in parts $V_i$ and $V_j$ (possibly with $i = j$) are adjacent if $M_{ij} = 1$ and nonadjacent if $M_{ij} = 0$. In particular, for $i = j$ these restrictions mean that $V_i$ is either a clique, or an independent set, or is unrestricted, when $M_{ii}$ is $1$, or $0$, or $\ast$, respectively. In general, some of the parts may be empty.

\medskip
\noindent{\bf Clique-width.} A \emph{$k$-graph} is a graph $G$ whose vertices are equipped  with labels from the set $[k]:=\{1,\ldots, k\}$. The \emph{clique-width} of a graph $G = (V,E)$, denoted $\cw(G)$, is the smallest positive integer $k$ such that there exists a $k$-graph isomorphic to $G$ that can be obtained as the result of a finite sequence of iteratively applying the following operations:
\begin{itemize}
  \item $i(v)$ where $i\in [k]$: creation of a $1$-vertex $k$-graph with its unique vertex $v$ labeled $i$,
  \item $G_1\oplus G_2$, disjoint union of two $k$-graphs $G_1$ and $G_2$,
  \item $\rho_{i\to j}(G')$ where $i,j\in [k]$ with $i\neq j$: the $k$-graph obtained from
  the $k$-graph $G'$ by assigning label $j$ to all vertices having label $i$
  \item $\eta_{i,j}(G')$ where $i,j\in [k]$ with $i\neq j$: the $k$-graph obtained from the $k$-graph $G'$ by adding to it all edges of the form $uv$ where $u$ has label $i$ and $v$ has label $j$.
\end{itemize}
An algebraic expression describing the above construction is called a \emph{ $k$-expression} of $G$. For example, the following is a $3$-expression of a $4$-vertex path with vertices $v_1,v_2,v_3,v_4$ and edges $v_1v_2,v_2v_3,v_3v_4$:
$$\eta_{2,3}(\rho_{3\to 2}(\rho_{2\to 1}(\eta_{2,3}(\eta_{1,2}(1(v_1)\oplus 2(v_2))\oplus 3(v_3))))\oplus 3(v_3))\,.$$
Given a $k$-expression $\tau$ of $G$, we denote by $|\tau|$ its length, which is the number of symbols one needs to write it down.

\subsection{Preliminaries on hypergraphs}

A \emph{hypergraph} $\mathcal{H}$ is a pair $(V, E) $ where $V = V(\mathcal{H})$ is a finite set of \emph{vertices} and $E = E(\mathcal{H})$ is a set of subsets of $V$, called \emph{hyperedges}.
A hypergraph is \emph{$k$-uniform} if all its hyperedges have size $k$.
In particular, graphs are precisely the $2$-uniform hypergraphs.
Every hypergraph $\mathcal{H}=(V,E)$ with a fixed pair of orderings
of its vertices and edges, say
$V = \{v_1,\ldots,  v_n\}$,
and
$E = \{e_1,\ldots,  e_m\}$,
can be represented with its \emph{incidence matrix} $A^\mathcal{H}\in \{0,1\}^{E\times V}$ having rows and columns indexed by edges and vertices of $\mathcal{H}$, respectively,
and being defined
as
\[
A^\mathcal{H}_{i,j} =
\left\{
\begin{array}{ll}
1, & \hbox{if $v_j\in e_i$;} \\
0, & \hbox{otherwise.}
\end{array}
\right.
\]

\noindent{\bf Sperner hypergraphs.} A hypergraph is said to be \emph{Sperner} if no hyperedge contains another one, that is, if $e,f\in E$ and $e\subseteq f$ implies $e = f$; see, e.g.,~\citet{MR1544925,MR0258642,MR0396277}. Sperner hypergraphs were studied in the literature under different names
including \emph{simple hypergraphs} by~\citet{MR1013569},
\emph{clutters} by~Billera~\cite{MR0307924,MR0307923} and by Edmonds and Fulkerson~\cite{MR0269433,MR0255235}, and \emph{coalitions} in the game theory literature~\cite{MR0219323}.

\medskip
\noindent{\bf Dually Sperner hypergraphs.} Sperner hypergraphs can be equivalently defined as the hypergraphs such that every two distinct hyperedges $e$ and $f$ satisfy $\min\{|e\setminus f|,|f\setminus e|\}\ge 1$. This point of view motivated Chiarelli and Milani\v c to define in~\cite{MR3281177} a hypergraph $\mathcal{H}$ to be \emph{dually Sperner} if every two distinct hyperedges $e$ and $f$ satisfy $\min\{|e\setminus f|,|f\setminus e|\}\le 1$. It was shown in~\cite{MR3281177} that dually Sperner hypergraphs are threshold.

\medskip
\noindent{\bf 1-Sperner hypergraphs.} Following~\citet{BGM-decomposing-1-Sperner}, we say that a hypergraph $\mathcal{H}$ is \emph{$k$-Sperner} if every two distinct hyperedges $e$ and $f$ satisfy $$1\le \min\{|e\setminus f|,|f\setminus e|\}\le k\,.$$
In particular, $\mathcal{H}$ is $1$-Sperner if every two distinct hyperedges $e$ and $f$ satisfy $$\min\{|e\setminus f|,|f\setminus e|\}= 1\,,$$
or, equivalently, if, for any two distinct hyperedges $e$ and $f$ of $\mathcal{H}$ with $|e| \leq |f|$, we have $|e\setminus f| = 1$.

As proved by~\citet{BGM-decomposing-1-Sperner}, $1$-Sperner hypergraphs admit a
recursive decomposition. In order to state the result, we first need to recall some definitions. The following operation produces (with one exception) a new $1$-Sperner hypergraph from a given pair of $1$-Sperner hypergraphs.

\begin{definition}[Gluing of two hypergraphs]
Given a pair of vertex-disjoint hypergraphs
$\mathcal{H}_1 = (V_1,E_1)$ and
$\mathcal{H}_2 = (V_2,E_2)$ and a new vertex $z\not\in V_1\cup V_2$,
the \emph{gluing of $\mathcal{H}_1$ and $\mathcal{H}_2$} is the hypergraph
$\mathcal{H} = \mathcal{H}_1\odot  \mathcal{H}_2$ such that
$$V(\mathcal{H}) = V_1\cup V_2\cup\{z\}$$ and
$$E(\mathcal{H}) = \{\{z\}\cup e\mid e\in E_1\} \cup \{V_1\cup e\mid e\in E_2\}\,.$$
\end{definition}

Let us illustrate the operation of gluing on an example, in terms of incidence matrices. Let $n_i = |V_i|$ and $m_i = |E_i|$ for $i = 1,2$, and let us denote by $\bbzero^{k,\ell}$, resp.~$\bbone^{k,\ell}$, the $k\times \ell$ matrix of all zeroes, resp.~of all ones. Then, the incidence matrix of the gluing of $\mathcal{H}_1$ and $\mathcal{H}_2$ can be written as
\[
A^{\mathcal{H}_1\odot \mathcal{H}_2} =\left(
    \begin{array}{ccc}
      \bbone^{m_1,1} & A^{\mathcal{H}_1} & \bbzero^{m_1, n_2} \\
      \bbzero^{m_2,1} & \bbone^{m_2, n_1}  & A^{\mathcal{H}_2} \\
    \end{array}
  \right)\,.
\]
See Fig.~\ref{fig:1} for an example.

\begin{figure}[h!]
  \centering
   \includegraphics[width=150mm]{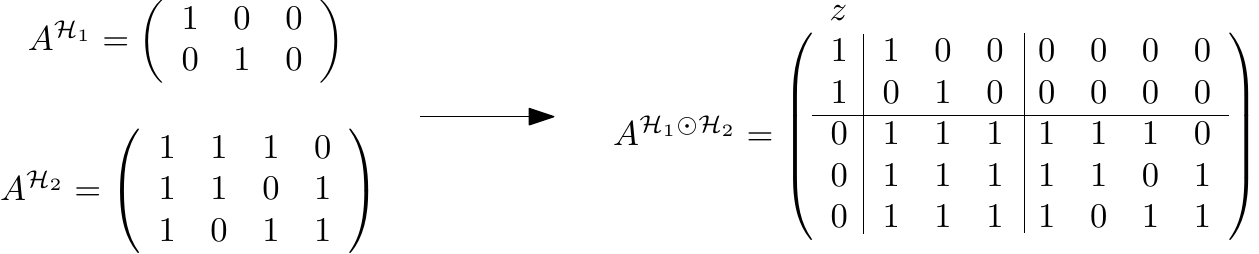}
  \caption{An example of gluing of two hypergraphs (from~\cite{BGM-decomposing-1-Sperner}).}\label{fig:1}
\end{figure}

\begin{observation}\label{obs:constituents}
If the gluing of $\mathcal{H}_1$ and $\mathcal{H}_2$ is a $1$-Sperner hypergraph, then $\mathcal{H}_1$ and $\mathcal{H}_2$ are also $1$-Sperner.
\end{observation}

Given a vertex $z$ of a hypergraph $\mathcal{H}$, we say that a hypergraph $\mathcal{H}$ is \emph{$z$-decomposable} if for every two hyperedges $e,f\in E(\mathcal{H})$ such that $z\in e\setminus f$, we have $e\setminus\{z\}\subseteq f$. Equivalently,
$\mathcal{H}$ is $z$-decomposable if its vertex set can be partitioned as $V(\mathcal{H}) = \{z\}\cup V_1\cup V_2$ such that $\mathcal{H} = \mathcal{H}_1\odot  \mathcal{H}_2$ for some hypergraphs $\mathcal{H}_1 = (V_1,E_1)$ and $\mathcal{H}_2 = (V_2,E_2)$.

The following is the main structural result about $1$-Sperner hypergraphs.

\begin{sloppypar}
\begin{theorem}[\citet{BGM-decomposing-1-Sperner}]\label{thm:decomposition}
Every $1$-Sperner hypergraph $\mathcal{H} = (V,E)$ with $V\neq \emptyset$
is $z$-decomposable for some $z\in V(\mathcal{H})$, that is, it is the gluing of two $1$-Sperner hypergraphs.
\end{theorem}
\end{sloppypar}

\medskip
\noindent{\bf Threshold hypergraphs.} A hypergraph $\mathcal{H} )= (V,E)$ is said to be \emph{threshold} if there exist a non-negative integer weight function $w:V\to \mathbb{Z}_{\ge 0}$ and a non-negative integer threshold $t\in \mathbb{Z}_{\ge 0}$ such that for every subset $X\subseteq V$, we have
$w(X)\ge t$ if and only if $e\subseteq X$ for some $e\in E$.

Threshold hypergraphs were defined in the uniform case by Golumbic~\cite{MR562306} and studied further by
Reiterman et al.~\cite{MR791660}. In their full generality (that is, without the restriction that the hypergraph is uniform), the concept of threshold hypergraphs is equivalent to that of threshold monotone Boolean functions, see, e.g., \cite{MR0439441}. A polynomial-time recognition algorithm for threshold monotone Boolean functions represented by their complete DNF was given by Peled and Simeone~\cite{MR798011}. The algorithm is based on linear programming and implies a polynomial-time recognition algorithm for threshold hypergraphs.

The mapping that takes every hyperedge $e\in {\cal E}$ to its \emph{characteristic vector} $\chi^e\in \{0,1\}^V$, defined by
\[
\chi^e_v = \left\{
\begin{array}{ll}
1, & \hbox{if $v\in e$} \\
0, & \hbox{otherwise\,,}
\end{array}
\right.
\]
shows that the sets of hyperedges of threshold Sperner hypergraphs are in a one-to-one correspondence with the sets of minimal feasible binary solutions of the linear inequality $w^\top x\ge t$. A set of vertices $X\subseteq V$ in a hypergraph is said to be \emph{independent} if it does not contain any hyperedge, and \emph{dependent} otherwise. Thus, threshold hypergraphs are exactly the hypergraphs admitting a linear function on the vertices separating the characteristic vectors of independent sets from the characteristic vectors of dependent sets.

For later use, we recall the following theorem.

\begin{theorem}[\citet{BGM-decomposing-1-Sperner,MR3281177}]\label{thm:1-Sperner-threshold}
Every $1$-Sperner hypergraph is threshold.
\end{theorem}

Theorem~\ref{thm:1-Sperner-threshold} was derived in~\cite{BGM-decomposing-1-Sperner} from Theorem~\ref{thm:decomposition}. However, note that every $1$-Sperner hypergraph is dually Sperner, which means that Theorem~\ref{thm:1-Sperner-threshold} also follows from the fact that dually Sperner hypergraphs are threshold, which was established earlier than Theorem~\ref{thm:decomposition} but in a non-constructive way~\cite{MR3281177}.

\medskip
\begin{sloppypar}
\noindent{\bf $k$-asummable hypergraphs.} A hypergraph is \emph{$k$-asummable} if it has no $k$ (not necessarily distinct) independent sets $A_1,\ldots, A_k$ and $k$ (not necessarily distinct) dependent sets $B_1,\ldots, B_k$ such that $$\sum_{i = 1}^k\chi^{A_i} = \sum_{i = 1}^k\chi^{B_i}\,.$$ A hypergraph is \emph{asummable} if it is $k$-asummable for every $k\ge 2$.

\begin{theorem}[Chow~\cite{Chow} and Elgot~\cite{5397278}]\label{thm:ChowElgot}
A hypergraph is threshold if and only if it is asummable.
\end{theorem}

The result was given in terms of monotone Boolean functions; more recent and complete information on this can be found in~\cite{MR2742439}.
For later use we state explicitly the simple fact that $2$-asummability is a necessary condition for thresholdness.

\begin{corollary}\label{obs:2-asummable-threshold}
Every threshold hypergraph is $2$-asummable.
\end{corollary}

It is also known that $2$-asummability does not imply thresholdness in general~\cite{MR2742439}, and thresholdness does not imply $1$-Spernerness (as can be seen by considering the edge set of the complete graph $K_4$).
\end{sloppypar}

\medskip
\noindent{\bf Transversal hypergraphs.}
Let $\mathcal{H} = (V,E)$ be a 
hypergraph. A \emph{transversal} of $\mathcal{H}$ is a set of vertices intersecting all hyperedges of $\mathcal{H}$. The \emph{transversal hypergraph} $\mathcal{H}^T$ is the hypergraph with vertex set $V$
in which a set $S\subseteq V$ is a hyperedge if and only if $S$ is an inclusion-minimal transversal of $\mathcal{H}$.

\begin{observation}[see, e.g., \citet{MR1013569}]\label{obs:transversals}
If $\mathcal{H}$ is a Sperner hypergraph, then $(\mathcal{H}^T)^T = \mathcal{H}$.
\end{observation}

A pair of a mutually transversal Sperner hypergraphs naturally corresponds to a pair of dual monotone Boolean functions, see~\cite{MR2742439}.

\begin{observation}[see, e.g., \citet{MR0439441}]\label{obs:transversal-k-asummable}
For every integer $k\ge 2$,
a Sperner hypergraph $\mathcal{H}$ is $k$-asummable if and only if
its transversal hypergraph $\mathcal{H}^T$ is $k$-asummable.
\end{observation}

\begin{corollary}[see, e.g., \citet{MR2742439}]\label{cor:transversal-threshold}
For every integer $k\ge 2$,
a Sperner hypergraph $\mathcal{H}$ is threshold if and only if
its transversal hypergraph $\mathcal{H}^T$ is threshold.
\end{corollary}

\medskip
\noindent{\bf Clique hypergraphs and conformal hypergraphs.} Given a graph $G$, the \emph{clique hypergraph of $G$} is the hypergraph $\mathcal{C}(G)$ with vertex set $V(G)$ in which the hyperedges are exactly the (inclusion-)maximal cliques of $G$. A hypergraph is $\mathcal{H}$ is \emph{conformal} if for every set $X\subseteq V(\mathcal{H})$ such that every pair of elements in $X$ is contained in a hyperedge, there exists a hyperedge containing $X$.

\begin{theorem}[\citet{MR1013569}]\label{thm:Berge}
A Sperner hypergraph is conformal if and only if it is the clique hypergraph of a graph.
\end{theorem}

\medskip
\noindent{\bf Interrelations between the considered classes of graphs and hypergraphs.} In Fig.~\ref{fig:Hasse}, we show the Hasse diagram of the partial order of the hypergraph classes mentioned above, along with some classes of graphs, ordered with respect to inclusion.
The diagram includes equalities between the following classes of hypergraphs:
\begin{itemize}
  \item threshold and asummable (Theorem~\ref{thm:ChowElgot});
  \item $2$-uniform threshold, $2$-uniform $2$-asummable, and threshold graphs (Theorem~\ref{thm:threshold-graphs});
  \item conformal $1$-Sperner, conformal Sperner threshold, conformal Sperner $2$-asummable hypergraphs, and clique hypergraphs of threshold graphs (Theorems~\ref{thm:Berge} and \ref{thm:threshold}).
\end{itemize}

\begin{figure}[h!]
  \centering
   \includegraphics[width=\textwidth]{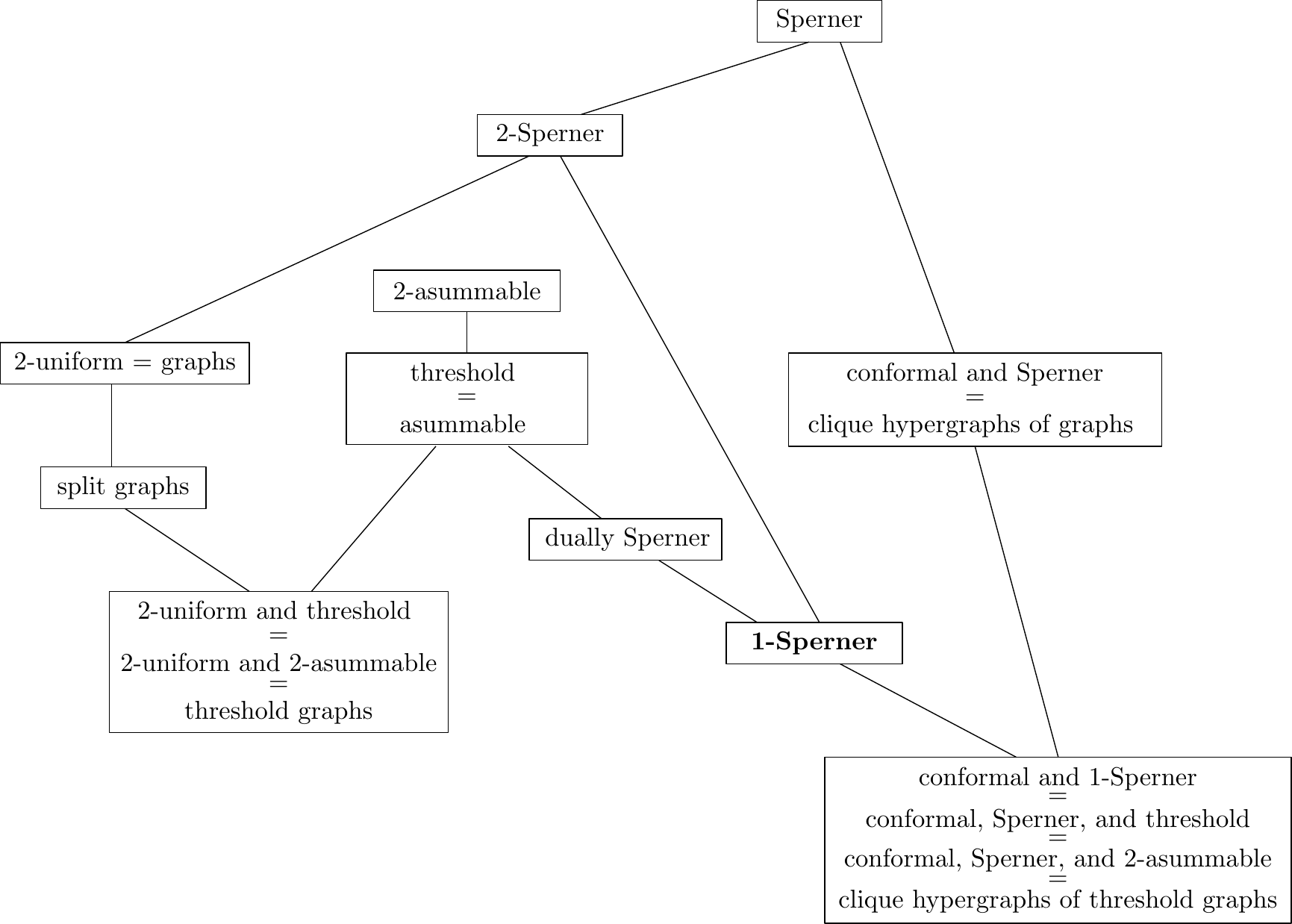}
  \caption{Inclusion relations between several classes of graphs and hypergraphs.}\label{fig:Hasse}
\end{figure}

The fact that every threshold graph is split follows from Theorem~\ref{thm:char2}, while the fact that every dually Sperner hypergraph is threshold was proved by~\citet{MR3281177}. The remaining inclusions are trivial or follow by transitivity.

Finally, the following examples show that all inclusions are strict and that there are no other inclusions (in particular, there are no other equalities):
\begin{itemize}
  \item the complete graph $K_3$ is a $2$-uniform $1$-Sperner hypergraph that it is not conformal;
  \item the complete graph $K_4$ is a threshold graph and, when viewed as a $2$-uniform hypergraph, it is threshold and Sperner but neither dually Sperner nor conformal;
  \item the path $P_4$ is a split graph but, as a $2$-uniform hypergraph,
      it is conformal and Sperner but not $2$-asummable;
  \item the clique hypergraph of the complete graph $K_3$ (which is threshold) is not $2$-uniform;
  \item the hypergraph with vertex set $\{1\}$ and hyperedge set $\{\emptyset,\{1\}\}$ is dually Sperner but not Sperner,
  \item the clique hypergraph of the graph $2K_3$ is conformal and Sperner but not $2$-Sperner;
 \item an example of a $2$-asummable non-threshold hypergraph is given, e.g., in~\cite{MR2742439,CM-ISAIM2014}.
\end{itemize}
The remaining non-inclusions follow by transitivity.

\section{Characterizations of threshold and domishold graphs}\label{sec:old-results}

In this section, we consider the characterizations of threshold and domishold graphs due to \citet{MR0479384} and~\citet{MR0491342} and show that they imply that for several families of hypergraphs derived from graphs, some or all of the three (generally properly nested) properties of $1$-Spernerness, thresholdness, and $2$-asummability coincide.

\subsection{Threshold graphs}\label{sec:threshold-old}

The first result along these lines follows easily from the characterizations of threshold graphs by forbidden induced subgraphs due to~Chv\'atal and Hammer, see e.g., Corollary 1D~\cite{MR0479384}.

\begin{theorem}[\citet{MR0479384}]\label{thm:threshold-graphs}
For a graph $G$, the following statements are equivalent:
\begin{enumerate}
  \item[(1)] $G$ is a threshold graph.
  \item[(2)] $G$ is a ($2$-uniform) threshold hypergraph.
  \item[(3)] $G$ is a ($2$-uniform) $2$-asummable hypergraph.
\end{enumerate}
\end{theorem}

%
%
%
%

To see that the list above cannot be extended with the property of $G$ being $1$-Sperner, consider, for example, the graph $K_4$. This is a threshold graph but it is not $1$-Sperner.

Threshold graphs have many other characterizations (see, e.g.,~\cite{MR1417258}). We now prove several more, which can also be seen as characterizations of certain classes of hypergraphs.

A \emph{vertex cover} of a graph $G$ is a set $S\subseteq V(G)$ such that every edge of $G$ has at least one endpoint in $S$. Given a graph $G$, the \emph{vertex cover hypergraph} $\mathcal{VC}(G)$ is the hypergraph with vertex set $V(G)$ in which a set $S\subseteq V(G)$ is a hyperedge if and only if $S$ is an inclusion-minimal vertex cover.
In contrast to~\Cref{thm:threshold-graphs}, if we focus on
characterizing thresholdness in terms of the vertex cover hypergraph,
then the property of $1$-Spernerness appears.

\begin{theorem}\label{thm:threshold-graphs-2}
For a graph $G$, the following statements are equivalent:
\begin{enumerate}
  \item[(1)] $G$ is threshold.
  \item[(2)] The vertex cover hypergraph $\mathcal{VC}(G)$ is $1$-Sperner.
  \item[(3)] The vertex cover hypergraph $\mathcal{VC}(G)$ is threshold.
  \item[(4)] The vertex cover hypergraph $\mathcal{VC}(G)$ is $2$-asummable.
\end{enumerate}
\end{theorem}

\begin{sloppypar}
\begin{proof}
The chain of implications $(2)\Rightarrow (3)\Rightarrow (4)$
follows from Theorem~\ref{thm:1-Sperner-threshold} and \Cref{obs:2-asummable-threshold}.
Since the vertex cover hypergraph $\mathcal{VC}(G)$ is Sperner and it is the transversal hypergraph of the $2$-uniform hypergraph $G$, \Cref{obs:transversals} implies that $(\mathcal{VC}(G))^T = G$.
Therefore, by~\Cref{obs:transversal-k-asummable}, $\mathcal{VC}(G)$ is $2$-asummable if and only if $G$ is $2$-asummable, which together with Theorem~\ref{thm:threshold-graphs} yields $(4)\Rightarrow (1)$.

For the implication $(1)\Rightarrow (2)$, suppose that the vertex cover hypergraph $\mathcal{VC}(G)$ is not $1$-Sperner. We will prove that $G$ is not threshold by showing that $G$ contains an induced subgraph isomorphic to
$P_4$, $C_4$, or $2K_2$. Since $\mathcal{VC}(G)$ is Sperner but not $1$-Sperner, there exist two minimal vertex covers $S$ and $S'$ of $G$ such that $\min\{|S\setminus S'|,|S'\setminus S|\}\ge 2$. Since $S$ is a vertex cover, the set $S'\setminus S$ is independent. Similarly, $S\setminus S'$ is independent. We claim that the bipartite subgraph of $G$
induced by $(S\setminus S')\cup(S'\setminus S)$
has no isolated vertices. By symmetry it suffices to show that every vertex $v\in S\setminus S'$ has a neighbor in $S'\setminus S$. If this were not the case, then by the minimality of $S$ we infer that $v$ would have a neighbor in $V\setminus (S\cup S')$. However, this would contradict the fact that $S'$ is a vertex cover. Since every bipartite graph without isolated vertices and having at least two vertices in both sides of a bipartition has an induced $P_4$, $C_4$, or $2K_2$, we conclude that $G$ is not threshold by \Cref{thm:char1}.
\end{proof}
\end{sloppypar}

We show next that in the class of conformal Sperner hypergraphs, the notion of thresholdness coincides with both $2$-asummability and $1$-Spernerness.
Moreover, it exactly characterizes threshold graphs. Recall that a Sperner hypergraph is conformal if and only if it is the clique hypergraph of a graph (\Cref{thm:Berge}).

\begin{theorem}\label{thm:threshold}
For a graph $G$, the following statements are equivalent:
\begin{enumerate}
  \item[(1)] $G$ is threshold.
  \item[(2)] The clique hypergraph $\mathcal{C}(G)$ is $1$-Sperner.
  \item[(3)] The clique hypergraph $\mathcal{C}(G)$ is threshold.
  \item[(4)] The clique hypergraph $\mathcal{C}(G)$ is $2$-asummable.
\end{enumerate}
\end{theorem}

\begin{proof}
For the implication $(1)\Rightarrow (2)$, suppose that the clique hypergraph $\mathcal{C}(G)$ is not $1$-Sperner.
Since $\mathcal{C}(G)$ is Sperner but not $1$-Sperner, there exist two maximal cliques $S$ and $S'$ of $G$ such that $\min\{|S\setminus S'|,|S'\setminus S|\}\ge 2$. Clearly, $S\setminus S'$ and $S'\setminus S$ are cliques.
By the maximality of $S$, every vertex in $S'\setminus S$ has a non-neighbor in $S\setminus S'$. Similarly, every vertex in $S\setminus S'$ has a non-neighbor in $S'\setminus S$. Thus, the complement of the subgraph of $G$ induced by $(S\setminus S')\cup(S'\setminus S)$ is a bipartite graph with at least two vertices in both sides and without isolated vertices.
Therefore, it contains an induced $P_4$, $C_4$, or $2K_2$, and since this family of graphs is closed under complementation, we conclude that $G$ is not threshold by \Cref{thm:char1}.


The chain of implications $(2)\Rightarrow (3)\Rightarrow (4)$
follows from Theorem~\ref{thm:1-Sperner-threshold}
and \Cref{obs:2-asummable-threshold}.

Now we prove the implication $(4)\Rightarrow (1)$. We break the proof into a small series of claims.

\begin{claim}\label{cl-1}
For any graph $G$, if $A,B\in \mathcal{C}(G)$, $a\in A\setminus B$, and $b\in B\setminus A$ such that $(a,b)\in E(G)$, then the set $(A\setminus \{a\})\cup\{b\}$ is an independent set of the hypergraph $\mathcal{C}(G)$.
\end{claim}

\begin{proof}
If $C\subseteq (A\setminus \{a\})\cup\{b\}$ is a clique of $G$, then $C\cup\{a\}$ is also a clique of $G$, and hence $C$ is not maximal.
\end{proof}

\begin{claim}\label{cl-2}
If $\mathcal{C}(G)$ is $2$-asummable, and  $A,B\in \mathcal{C}(G)$, $a\in A\setminus B$, and $b\in B\setminus A$ then $(a,b)\not\in E(G)$.
\end{claim}
\begin{proof}
If $(a,b)\in E(G)$ then Claim \ref{cl-1} and the equality
\[
\chi^A+\chi^B ~=~ \chi^{(A\setminus \{a\})\cup\{b\}} + \chi^{(B\setminus\{b\})\cup\{a\}}
\]
contradicts the $2$-asummability of $\mathcal{C}(G)$.
\end{proof}

\begin{claim}\label{cl-3}
If $\mathcal{C}(G)$ is $2$-asummable, and  $A,B\in \mathcal{C}(G)$, $a\in A\setminus B$, and $b\in B\setminus A$, then the set $(A\cup B)\setminus \{a,b\}$ is an independent set of the hypergraph $\mathcal{C}(G)$.
\end{claim}
\begin{proof}
If $C\subseteq (A\cup B)\setminus \{a,b\}$ is a maximal clique of $G$, then
\begin{itemize}
\item $C\cup \{a\}$ is not a clique, and hence there exists a vertex $u\in (B\cap C)\setminus A$ (such that $\{a,u\}\not\in E(G)$);
\item $C\cup \{b\}$ is not a clique, and hence there exists a vertex $v\in (A\cap C)\setminus B$ (such that $\{b,v\}\not\in E(G)$).
\end{itemize}
Consequently we have vertices $u\in B\setminus A$ and $v\in A\setminus B$ such that $\{u,v\}\subseteq C$ implying $\{u,v\}\in E(G)$. This contradicts Claim \ref{cl-2}, which proves that such a maximal clique $C$ cannot exist.
\end{proof}

\begin{claim}\label{cl-4}
If $\mathcal{C}(G)$ is $2$-asummable then $G$ is a threshold graph.
\end{claim}

\begin{proof}
Let us recall the forbidden subgraph characterization of threshold graphs by Theorem \ref{thm:char1}, and assume indirectly that there exist four distinct vertices $\{a,b,c,d\}\subseteq V(G)$ such that $\{a,b\}, \{c,d\}\in E(G)$ and $\{a,c\}, \{b,d\}\not\in E(G)$ (that is, that $\{a,b,c,d\}$ induces either a $P_4$, a $C_4$, or a $2K_2$). Let us then consider maximal cliques $A\supseteq \{a,b\}$ and $B\supseteq \{c,d\}$ of $G$. Since $\{a,b\}, \{c,d\}\in E(G)$ such maximal cliques do exist.

Our assumptions that $\{a,c\}, \{b,d\}\not\in E(G)$ imply that $\{a,b\}\subseteq A\setminus B$ and $\{c,d\}\subseteq B\setminus A$.

Then we can apply Claim \ref{cl-3} and obtain that both $(A\cup B)\setminus\{a,c\}$ and $(A\cup B)\setminus\{b,d\}$ are independent sets of the hypergraph $\mathcal{C}(G)$. Since $(A\cap B)\cup\{a,c\}\subseteq (A\cup B)\setminus\{b,d\}$, it also follows that $(A\cap B)\cup\{a,c\}$ is an independent set of $\mathcal{C}(G)$. Consequently, the equality
\[
\chi^A+\chi^B~=~ \chi^{(A\cup B)\setminus\{a,c\}} + \chi^{(A\cap B)\cup\{a,c\}}
\]
contradicts the $2$-asummability of $\mathcal{C}(G)$. This contradiction proves that such a set of four vertices cannot exists, from which the claim follows by Theorem \ref{thm:char1}.
\end{proof}

The last claim proves the implication $(4)\Rightarrow (1)$, completing the proof of the theorem.
\end{proof}

\begin{remark}
The \emph{independent set hypergraph} of a graph $G$ is the hypergraph with vertex set $V(G)$ and in which the hyperedges are exactly the maximal independent sets of $G$. Since the class of threshold graphs is closed under taking complements, one could obtain three further characterizations of threshold graphs by replacing the clique hypergraph with the independent set hypergraph in any of the properties $(2)$--$(4)$ in Theorem~\ref{thm:threshold}.
\end{remark}

Let us finally show that the property of $1$-Spernerness is related to the family of threshold graphs also in the following way.
Given a hypergraph $\mathcal{H}$, its \emph{co-occurence graph} is the graph with vertex set $V(\mathcal{H})$ in which two vertices are adjacent if and only if there exists a hyperedge of $\mathcal{H}$ containing both (see, e.g., \cite[Chapter 10]{MR2742439}).

\begin{theorem}\label{thm:threshold-graphs-new}
A graph $G$ is threshold if and only if it is the co-occurence graph of some $1$-Sperner hypergraph.
\end{theorem}

\begin{proof}
If $G$ is threshold, then by \Cref{thm:threshold} the clique hypergraph $\mathcal{C}(G)$ is $1$-Sperner. Since two vertices of $G$ are adjacent if and only if they belong to some maximal clique, the co-occurrence graph of $\mathcal{C}(G)$ is $G$ itself. This shows the forward implication.

Suppose now that $\mathcal{H} = (V,E)$ is a $1$-Sperner hypergraph and let $G$ be its co-occurence graph. We prove that $G$ is threshold using induction on $n = |V(G)|$. For $n \le 1$, the statement is trivial, so let us assume $n>1$. We may also assume that $E\neq \emptyset$, since otherwise $G$ is edgeless and thus threshold.
If $\mathcal{H}$ has an isolated vertex, then by deleting it we get again a $1$-Sperner hypergraph. Furthermore, to the co-occurence graph of the resulting hypergraph we can add back an isolated vertex to obtain $G$.
Since by \Cref{cor:threshold}, this operations preserves thresholdness,
we may assume that $\mathcal{H}$ has no isolated vertices.

Since $\mathcal{H}$ is a $1$-Sperner hypergraph with $V \neq \emptyset$, Theorem~\ref{thm:decomposition} implies the existence of a vertex $z\in I$ such that $\mathcal{H}$ is $z$-decomposable. Let $z\in V$ denote such a vertex and let $\mathcal{H}$ be the corresponding gluing of $\mathcal{H}_1 = (V_1,E_1)$ and $\mathcal{H}_2 = (V_2,E_2)$.
That is, $V(\mathcal{H}) = V_1\cup V_2\cup\{z\}$ and $E = \{\{z\}\cup e\mid e\in E_1\} \cup \{V_1\cup e\mid e\in E_2\}$.
We may assume without loss of generality that $\mathcal{H}_1$ has no isolated vertex. For $i\in \{1,2\}$, let $G_i$ be the co-occurrence graph of $\mathcal{H}_i$. By the induction hypothesis, $G_1$ and $G_2$ are threshold graphs (note that this is the case also if some of the sets $V_1,V_2,E_1,E_2$ is empty).

Suppose first that $E_2 = \emptyset$. In this case, $E_1\neq\emptyset$ and since $\mathcal{H}$ has no isolated vertices, \hbox{$V_2= \emptyset$}.
Therefore, $G$ is the graph obtained from $G_1$ by adding to it the universal vertex $z$. Thus, by \Cref{cor:threshold}, $G$ is threshold.

Assume finally that $E_2 \neq \emptyset$.
Since $\mathcal{H}$ has no isolated vertices, $E_1\neq \emptyset$.
It is now not difficult to see that we can obtain $G$ in two steps:
first, we add to $G_2$ vertex $z$ as an isolated vertex; second, to the so obtained graph we add the vertices of $V_1$ (one by one, in any order) as universal vertices. Since $G_2$ is threshold, it follows from \Cref{cor:threshold} that $G$ is threshold. This completes the proof.
\end{proof}

\subsection{Domishold graphs}

Similar characterizations as for threshold graphs can be derived for domishold graphs.
Note that in the proof of~\Cref{thm:threshold-graphs-2} we used the fact that
the vertex cover hypergraph is transversal to the family of edges of
the graph. In the case of domishold graphs, we first consider
the hypergraph transversal to the family of (inclusion-minimal) dominating sets.
Given a graph $G$, the \emph{closed neighborhood hypergraph} $\mathcal{N}(G)$ is the hypergraph with vertex set $V(G)$ in which a set $S\subseteq V(G)$ is a hyperedge if and only if $S$ is an inclusion-minimal set of the form $N[v]$ for some $v\in V(G)$.

\begin{theorem}\label{thm:domishold}
For a graph $G$, the following statements are equivalent:
\begin{enumerate}
  \item[(1)] $G$ is domishold.
  \item[(2)] The closed neighborhood hypergraph $\mathcal{N}(G)$ is $1$-Sperner.
  \item[(3)] The closed neighborhood hypergraph $\mathcal{N}(G)$ is threshold.
  \item[(4)] The closed neighborhood hypergraph $\mathcal{N}(G)$ is $2$-asummable.
\end{enumerate}
\end{theorem}

\begin{proof}
For the implication $(1)\Rightarrow (2)$, suppose that the closed neighborhood hypergraph $\mathcal{N}(G)$ is not $1$-Sperner.
Since $\mathcal{N}(G)$ is Sperner but not $1$-Sperner,
there exist two vertices $u,v\in V(G)$ such that
\hbox{$\min\{|N[u]\setminus N[v]|,|N[v]\setminus N[u]|\}\ge 2$}.
Clearly, $u\neq v$. We consider two cases depending on whether $u$ and $v$ are adjacent or not.

Suppose first that $uv\not\in E(G)$. Then there exist vertices $u',v'\in V(G)$ such that $u'\in N(u)\setminus N[v]$ and
$v'\in N(v)\setminus N[u]$. Then, $\{u,v,u',v'\}$ must induce either a $P_4$ (if $u'$ and $v'$ are adjacent) or a $2K_2$ (otherwise).

Suppose now that $uv\in E(G)$. Then there exist four distinct vertices
$u',u'',v',v''$ in $V(G)\setminus \{u,v\}$ such that
$u',u''\in N(v)\setminus N[u]$
and
$v',v''\in N(u)\setminus N[v]$.
If $G$ does contains an induced $P_4$, then $G$ is not domishold by \Cref{thm:char3}. We may therefore assume that $G$ does not contain an induced $P_4$. This implies that every vertex of $\{u,u',u''\}$ is adjacent to every vertex of $\{v,v',v''\}$.
Since $u$ is adjacent to neither $u'$ nor $u''$ and
$v$ is adjacent to neither $v'$ nor $v''$, the subgraph of $G$ induced by
these six vertices is isomorphic to one of
$K_{3,3}$, $K_{3,3}^+$, $\overline{2P_3}$ (see Fig.~\ref{fig:domishold}).

In both cases, we conclude that $G$ is not domishold by \Cref{thm:char3}.


The chain of implications $(2)\Rightarrow (3)\Rightarrow (4)$
follows from Theorem~\ref{thm:1-Sperner-threshold}
and \Cref{obs:2-asummable-threshold}.

\medskip
To prove the implication $(4)\Rightarrow (1)$, we prove the contrapositive, that is, if $G$ is not domishold, then $\mathcal{N}(G)$ is not a $2$-asummable hypergraph. Suppose that $G$ is not domishold. By Theorem~\ref{thm:char3} $G$ has an induced subgraph $F$ isomorphic to a graph from the set $\{P_4,2K_2,K_{3,3},K_{3,3}^+,\overline{2P_3}\}$ (see Fig.~\ref{fig:domishold}).

Suppose first that $F$ is isomorphic to $P_4$ or to $2K_2$. Then there exist four distinct vertices $\{a,b,c,d\}\subseteq V(G)$ such that
$\{a,b\}, \{c,d\}\in E(G)$ and $\{a,c\}, \{a,d\}, \{b,c\}\not\in E(G)$.
Consider the sets $B_1 = N[a]$, $B_2 = N[b]$,
$A_1 = (B_1\setminus \{a\})\cup \{c\}$, and
$A_2 = (B_2\setminus \{c\})\cup \{b\}$.
It is clear that $B_1$ and $B_2$ are dependent sets of $\mathcal{N}(G)$.
The sets $A_1$ and $A_2$, however, are independent. By symmetry, it suffices to show that $A_1$ is independent in $\mathcal{N}(G)$, that is, that it contains no closed neighborhood of $G$. We have $N[c]\nsubseteq A_1$ since
$d\in N[c]\setminus A_1$. Furthermore, for every $u\in  A_1\setminus \{c\}$ we have $N[u]\nsubseteq A_1$ since $a\in N[u]\setminus A_1$. Finally, for every $u\in V(G)\setminus A_1$ we have $N[u]\nsubseteq A_1$ since $u\in N[u]\nsubseteq A_1$. Therefore, $A_1$ is independent.
But now, since  we have $\chi^{A_1}+\chi^{A_2} =  \chi^{B_1}  + \chi^{B_2}$,
$G$ is not $2$-asummable, which is what we wanted to prove.

Suppose now that $F$ is isomorphic to $K_{3,3}$, $K_{3,3}^+$, or $\overline{2P_3}$. We may assume that $G$ is $P_4$-free, since otherwise can repeat the arguments above. The fact that $F$ is an induced subgraph of $G$ implies the existence of a set of six distinct vertices $\{a,b,c,d,e,f\}\subseteq V(G)$ such that every vertex from $\{a,b,c\}$ is adjacent to every vertex from $\{d,e,f\}$, vertex $b$ is not adjacent to $a$ or $c$, and vertex
$e$ is not adjacent to $d$ and $f$, see Fig.~\ref{fig:domishold-proof}.

\begin{figure}[h!]
  \centering
   \includegraphics[width=28mm]{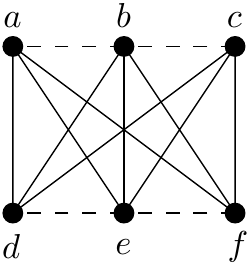}
\caption{Six vertices forming a forbidden induced subgraph for the class of domishold graphs.}
\label{fig:domishold-proof}
\end{figure}

Consider the sets $B_1 = N[b]$, $B_2 = N[e]$,
$A_1 = (B_1\setminus \{f\})\cup \{a\}$, and
$A_2 = (B_2\setminus \{a\})\cup \{f\}$.
It is clear that $B_1$ and $B_2$ are dependent sets of $\mathcal{N}(G)$.
Since $\chi^{A_1}+\chi^{A_2} =  \chi^{B_1}  + \chi^{B_2}$,
we can infer that $G$ is not $2$-asummable (which is what we want to prove).
if the sets $A_1$ and $A_2$ are independent.
So we may assume that $A_1$, say, is not independent in $\mathcal{N}(G)$.
This means that there exists a vertex $u\in V(G)$ such that
$N[u]\subseteq A_1$. Then $u\neq a$ since $f\in N[a]\setminus A_1$,
and similarly $u\neq b$.
Since $u\in A_1$ and every vertex in $A_1$ is either $a$ or belongs to $N[b]$, we infer that $u$ is adjacent to $b$. However, since $N[u]\subseteq A_1$, vertices $c$ and $f$ are not adjacent to $u$. But now, the vertex set
$\{u,b,f,c\}$ induces a $P_4$ in $G$, a contradiction.
\end{proof}

Given a graph $G$, the \emph{dominating set hypergraph} $\mathcal{D}(G)$ is the hypergraph with vertex set $V(G)$ in which a set $D\subseteq V(G)$ is a hyperedge if and only if $D$ is an inclusion-minimal dominating set of $G$.
The analogue of Theorem~\ref{thm:threshold-graphs} for domishold graphs
is the following.

\begin{theorem}\label{thm:domishold-2}
For a graph $G$, the following statements are equivalent:
\begin{enumerate}
  \item[(1)] $G$ is domishold.
  \item[(2)] The dominating set hypergraph $\mathcal{D}(G)$ is threshold.
  \item[(3)] The dominating set hypergraph $\mathcal{D}(G)$ is $2$-asummable.
\end{enumerate}
\end{theorem}

\begin{sloppypar}
\begin{proof}
Since the dominating set hypergraph $\mathcal{D}(G)$ is Sperner and it is the transversal hypergraph of the closed neighborhood hypergraph $\mathcal{N}(G)$, \Cref{obs:transversals} implies that $(\mathcal{D}(G))^T = \mathcal{N}(G)$. Therefore, by~\Cref{cor:transversal-threshold} and~\Cref{obs:transversal-k-asummable}, the dominating set hypergraph $\mathcal{D}(G)$ is threshold, resp.~$2$-asummable, if and only if the closed neighborhood hypergraph $\mathcal{N}(G)$ is threshold, resp.~$2$-asummable.
The theorem now follows from Theorem~\ref{thm:domishold}.
\end{proof}
\end{sloppypar}

To see that the list from Theorem~\ref{thm:domishold-2} cannot be extended with the property of
the dominating set hypergraph $\mathcal{D}(G)$ being $1$-Sperner, consider, for example, the graph
$C_4$. This is a domishold graph but its dominating set hypergraph is isomorphic to
$K_4$, which is not $1$-Sperner.

\subsection{Total domishold and connected-domishold graphs}

We conclude this section by stating two related theorems due to~\citet{MR3281177,CM-ISAIM2014} about the structure of graphs defined similarly as domishold graphs, but with respect to total, resp.~connected domination. A graph $G = (V,E)$ is said to be \emph{total domishold} (resp.,
\emph{connected-domishold}) if there exist a non-negative integer weight function $w:V\to \mathbb{Z}_{\ge 0}$ and a non-negative integer threshold $t\in \mathbb{Z}_{\ge 0}$ such that for every subset $X\subseteq V$, we have
$w(X)\ge t$ if and only if $X$ is a total dominating set (resp., a connected dominating set). Given a graph $G$, the \emph{neighborhood hypergraph} is the hypergraph with vertex set $V(G)$ in which a set $S\subseteq V(G)$ is a hyperedge if and only if $S$ is an inclusion-minimal set of the form $N(v)$ for some $v\in V(G)$. A \emph{cutset} in a graph $G$ is a set $S\subseteq V(G)$ such that $G-S$ is disconnected. The \emph{cutset hypergraph} is the hypergraph with vertex set $V(G)$ in which a set $S\subseteq V(G)$ is a hyperedge if and only if $S$ is an inclusion-minimal cutset in $G$.

\begin{theorem}[\citet{MR3281177}]\label{thm:total-domishold}
For a graph $G$, the following statements are equivalent:
\begin{enumerate}
  \item[(1)] Every induced subgraph of $G$ is total domishold.
  \item[(2)] The neighborhood hypergraph of every induced subgraph of $G$ is $1$-Sperner.
  \item[(3)] The neighborhood hypergraph of every induced subgraph of $G$ is threshold.
  \item[(4)] The neighborhood hypergraph of every induced subgraph of $G$ is $2$-asummable.
\end{enumerate}
\end{theorem}

\begin{theorem}[\citet{CM-ISAIM2014}]\label{thm:connected-domishold}
For a graph $G$, the following statements are equivalent:
\begin{enumerate}
  \item[(1)] Every induced subgraph of $G$ is connected-domishold.
  \item[(2)] The cutset hypergraph of every induced subgraph of $G$ is $1$-Sperner.
  \item[(3)] The cutset hypergraph of every induced subgraph of $G$ is threshold.
  \item[(4)] The cutset hypergraph of every induced subgraph of $G$ is $2$-asummable.
\end{enumerate}
\end{theorem}

\begin{sloppypar}
Characterizations of graphs every induced subgraph of which is total domishold, resp.~connected-domishold, in terms of forbidden induced subgraphs were also given in~\cite{MR3281177,CM-ISAIM2014}. We omit them here.
\end{sloppypar}

\section{Bipartite and split representations of $1$-Sperner hypergraphs}\label{sec:to-graphs}

In the remainder of the paper, we apply the decomposition theorem for $1$-Sperner hypergraphs (\Cref{thm:decomposition}) to derive decomposition theorems for four classes of graphs and derive some consequences.
The four classes come in two pairs: a class of split graphs and their complements (which are also split), and a class of bipartite graphs and their complements (which are cobipartite).
For our purpose it will be convenient to consider split graphs as already equipped with a split partition and bipartite graphs (bigraphs) as already equipped with a bipartition. Partly following the terminology of Dabrowski and Paulusma~\cite{MR3442572}, we will say that a \emph{labeled split graph} is a triple $(G,K,I)$ such that $G$ is a split graph with a split partition $(K,I)$. Similarly, a \emph{labeled bigraph} is a triple $(G,A,B)$ such that $G$ is a bigraph with a bipartition $(A,B)$. Given two labeled split graphs $(G_1,K^1,I^1)$ and $(G_2,K^2,I^2)$, we say that $G_1$ \emph{is (isomorphic to) an induced subgraph of} $G_2$ if there exists a pair of injective mappings $f:K^1\to K^2$ and $g:I^1\to I^2$ such that for all $(u,v)\in (K^1,I^1)$, vertices $a$ and $b$ are adjacent
in $G_1$ if and only if vertices $f(u)$ and $g(v)$ are adjacent in $G_2$.
If $G_1$ is not an induced subgraph of $G_2$, we say that
$G_2$ is \emph{$G_1$-free}. For labeled bigraphs, the definitions are similar.

To each hypergraph we can associate three types of \emph{incidence graphs} in a natural way: one labeled bipartite graph and two labeled split graphs.

\begin{definition}
Given a hypergraph $\mathcal{H} = (V,E)$, we define:
\begin{itemize}
  \item the \emph{bigraph} of $\mathcal{H}$ as the labeled bigraph $(G,V,E)$ such that $v\in V$ is adjacent to $e\in E$ if and only if $v\in e$,
  \item the \emph{vertex-clique split graph} of $\mathcal{H}$ as the labeled split graph $(G,V,E)$ such that $V$ is a clique and $v\in V$ is adjacent to $e\in E$ if and only if $v\in e$, and
  \item the \emph{edge-clique split graph} of $\mathcal{H}$ as the labeled split graph $(G,E,V)$ such $E$ is a clique and $v\in V$ is adjacent to $e\in E$ if and only if $v\in e$.
\end{itemize}
\end{definition}

See Fig.~\ref{fig:example2} for an example.

\begin{figure}[h!]
  \centering
   \includegraphics[width=0.97\textwidth]{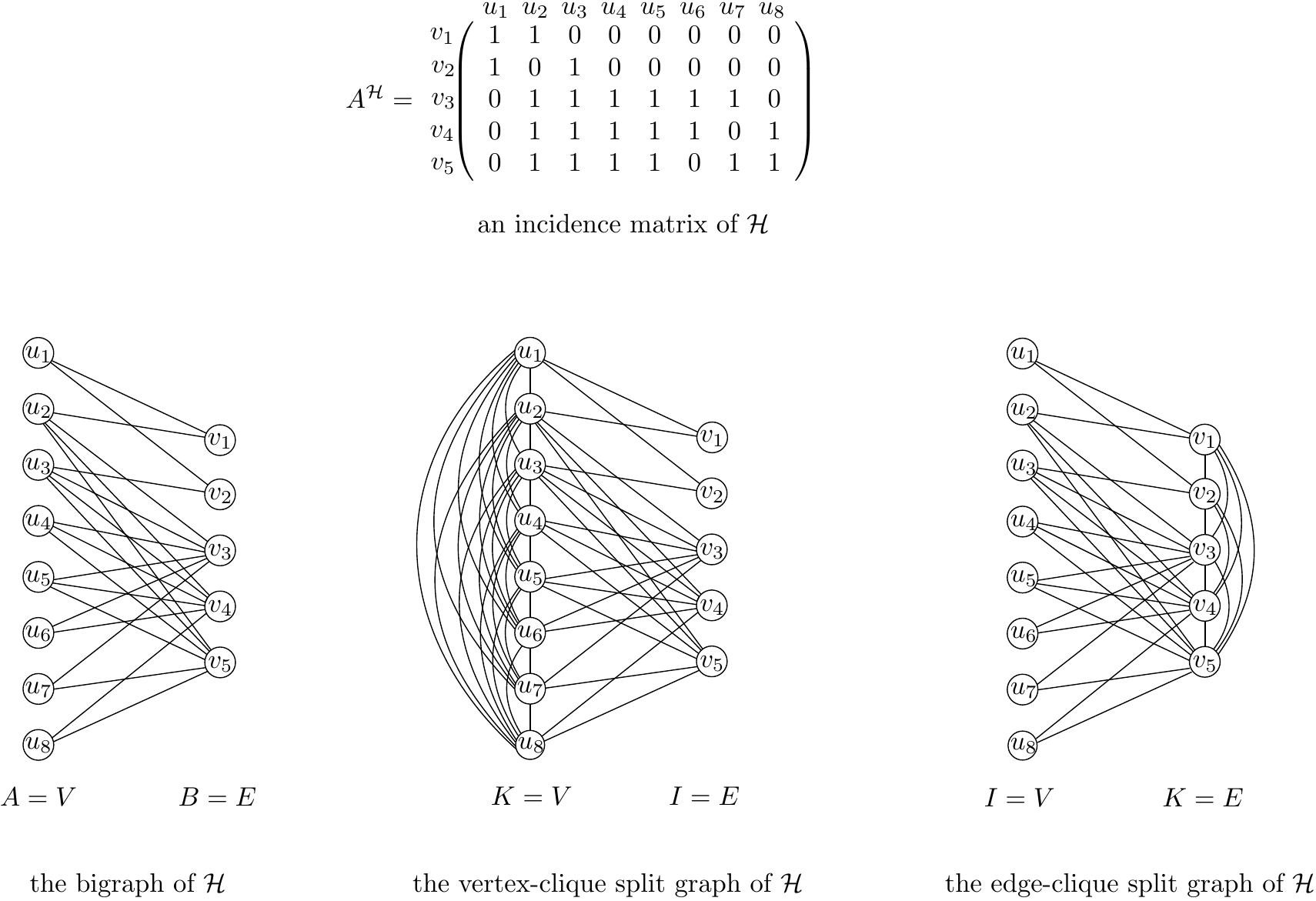}
\caption{An incidence matrix $A^{\mathcal{H}}$ of a hypergraph $\mathcal{H} = (V,E)$ with $V = \{u_1,\ldots, u_8\}$ and $E = \{v_1,\ldots, v_5\}$, and its three incidence graphs.}  \label{fig:example2}
\end{figure}

Representing a hypergraph with its bigraph is standard in studies of general~\cite{MR1013569} or highly regular~\cite{MR0006834,MR0038078,MR2978043} hypergraphs. Representing a hypergraph with a split incidence graph has also been used in the literature, e.g., by~\citet{MR3276423}, followed by~\citet{MR3281177,CM-ISAIM2014}, who considered the vertex-clique split graph (using slightly different terminology).

\medskip
We will characterize $1$-Sperner hypegraphs in terms of properties of their bigraphs and split graphs. Recall that a hypergraph $\mathcal{H}$ is said to be dually Sperner if every two distinct hyperedges $e$ and $f$ satisfy
$\min\{|e\setminus f|,|f\setminus e|\}\le 1$, and that
$\mathcal{H}$ is $1$-Sperner if and only if it is Sperner and dually Sperner. Both the Sperner and the dually Sperner property can be easily translated to properties of their incidence graphs. The corresponding characterizations of $1$-Sperner hypegraphs will then follow as corollaries.

\medskip
The Sperner property is naturally expressed in terms of constraints on the neighborhoods of vertices in one of the parts of the bipartition of a labeled bigraph (resp., split partition of a labeled split graph) intersected with the other part.

\begin{definition}\label{def:Sperner-etc}
A labeled bigraph $(G,A,B)$ is said to be \emph{right-Sperner} if for all $u,v\in B$, if $N(u)\subseteq N(v)$ then $u = v$.
A labeled split graph $(G,K,I)$ is said to be \emph{clique-Sperner} if for all $u,v\in K$, if $N(u)\cap I\subseteq N(v)\cap I$ then $u = v$, and \emph{independent-Sperner} if for all $u,v\in I$, if $N(u)\subseteq N(v)$ then $u = v$.
\end{definition}
\begin{observation}\label{obs:Sperner}
For every hypergraph $\mathcal{H} = (V,E)$, the following conditions are equivalent:
\begin{enumerate}
  \item $\mathcal{H}$ is Sperner.
  \item The bigraph of $\mathcal{H}$ is right-Sperner.
  \item The vertex-clique split graph of $\mathcal{H}$ is \hbox{independent-Sperner}.
  \item The edge-clique split graph of $\mathcal{H}$ is clique-Sperner.
\end{enumerate}
\end{observation}

\medskip
To express the dually Sperner property of a hypergraph in terms of properties of its incidence graphs, we introduce three small graphs, namely the $2P_3$, the graph denoted by $H$ and obtained from $2P_3$ by adding the edge between the two vertices of degree two, and by $\overline H$ the
(graph theoretic) complement of $H$. When necessary, we will consider the bigraph $2P_3$ as a labeled bigraph, with respect to the bipartition shown in Fig.~\ref{fig:small-graphs}, and the split graphs $H$ and $\overline{H}$ as labeled split graphs, with respect to their unique split partitions.

\begin{figure}[h!]
  \centering
   \includegraphics[width=90mm]{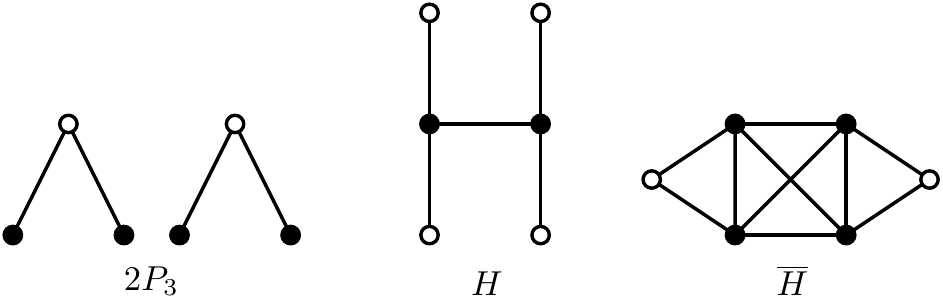}
\caption{The graphs $2P_3$, $H$, and $\overline{H}$. We consider $2P_3$ as a labeled bigraph with bipartition $(A,B)$ containing the four black vertices in $A$ and the two white vertices in $B$. The black, resp.~white vertices of $H$ and of $\overline{H}$ denote sets $K$ and $I$ in their (unique) split partitions.}  \label{fig:small-graphs}
\end{figure}

\begin{remark}
When the bigraph $2P_3$ and the labeled split graphs $H$ and $\overline{H}$ are considered labeled, as shown in Fig.~\ref{fig:FF}, they are the bigraph, the edge-clique split graph, and the vertex-clique split graph, respectively, of the smallest non-$1$-Sperner hypergraph.

\begin{figure}[h!]
  \centering
   \includegraphics[width=85mm]{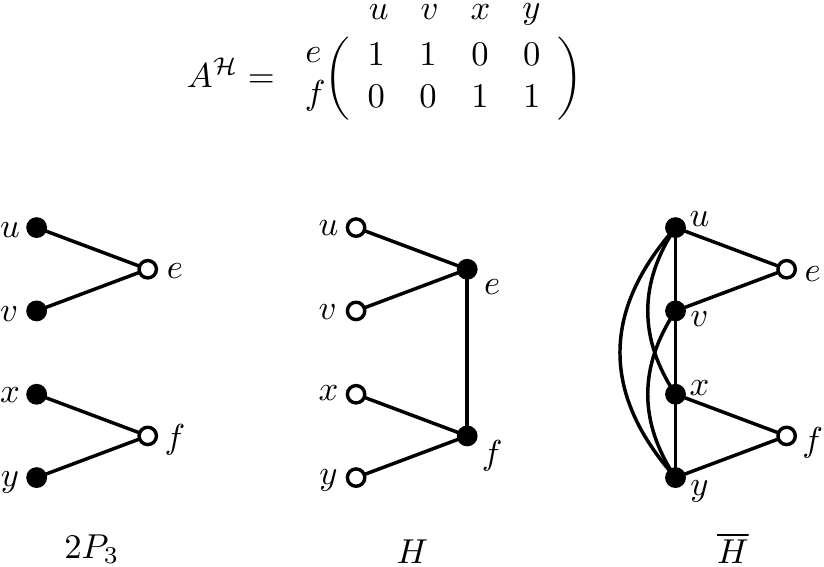}
\caption{An incidence matrix $A^{\mathcal{H}}$ of a non-$1$-Sperner hypergraph $\mathcal{H} = (V,E)$ with $V = \{u,v,x,y\}$ and $E = \{e,f\}$, and its three incidence graphs. They are isomorphic to $2P_3$, $H$, and $\overline{H}$, respectively.}  \label{fig:FF}
\end{figure}
\end{remark}

The dually Sperner property of a hypergraph can be expressed in terms of properties of its incidence graphs as follows.

\begin{observation}\label{obs:dually-Sperner}
For every hypergraph $\mathcal{H} = (V,E)$, the following conditions are equivalent:
\begin{enumerate}
  \item $\mathcal{H}$ is dually Sperner.
  \item The bigraph of $\mathcal{H}$ is $2P_3$-free in the labeled sense.
  \item The vertex-clique split graph of $\mathcal{H}$ is $\overline{H}$-free.
  \item The edge-clique split graph of $\mathcal{H}$ is $H$-free.
\end{enumerate}
\end{observation}

Expressing the $1$-Sperner property of a hypergraph in terms of properties of its incidence graphs now follows.

\begin{corollary}\label{prop:translation-to-graphs}
For every hypergraph $\mathcal{H} = (V,E)$, the following conditions are equivalent:
\begin{enumerate}
  \item $\mathcal{H}$ is $1$-Sperner.
  \item The bigraph of $\mathcal{H}$ is right-Sperner and $2P_3$-free in the labeled sense.
  \item The vertex-clique split graph of $\mathcal{H}$ is \hbox{independent-Sperner} and $\overline{H}$-free.
  \item The edge-clique  split graph of $\mathcal{H}$ is clique-Sperner and $H$-free.
\end{enumerate}
\end{corollary}

\begin{proof}
Immediate from Observations~\ref{obs:Sperner} and~\ref{obs:dually-Sperner}, using the fact that $\mathcal{H}$ is $1$-Sperner if and only if it is both Sperner and dually Sperner.
\end{proof}

\section{Decomposition theorems for four graph classes}\label{sec:structure-graphs}

We will now apply the connection between $1$-Sperner hypergraphs and their incidence graphs explained in \Cref{sec:to-graphs} to derive decomposition theorems for a class of split graphs and their complements, and a class of bigraphs and their complements.
To this end, we need to generalize Definition~\ref{def:Sperner-etc} to general (not labeled) bigraphs and split graphs. A bigraph $G$ is said to be \emph{right-Sperner} if it has a bipartition $(A,B)$ such that the labeled bigraph $(G,A,B)$ is right-Sperner. A split graph $G$ is said to be \emph{clique-Sperner} (resp., \emph{\hbox{independent-Sperner}}) if it has a split partition $(K,I)$ such that the labeled split graph $(G,K,I)$ is clique-Sperner (resp., \hbox{independent-Sperner}).

We first prove the decomposition theorem for the class of
$H$-free clique-Sperner split graphs. We start with a useful lemma.

\begin{lemma}\label{lem:split}
Suppose that a split graph $G$ has a split partition $(K,I)$ such that the labeled split graph $(G,K,I)$ is clique-Sperner. Then, the partition $(K,I)$ has the following properties:
\begin{itemize}
  \item if $G$ is edgeless, then $|K|\le 1$,
  \item if $|E(G)| = 1$, then $K=\{v\}$ for some non-isolated vertex $v\in V(G)$, and
  \item if $G$ has at least two edges, then $(K,I)$ is the unique split partition of $G$ such that $I$ is a maximal independent set in $G$.
\end{itemize}
Furthermore, there is an $\mathcal{O}(|V(G)|^3)$ algorithm that computes a split partition $(K,I)$ of a given split graph $G$ such that the labeled split graph $(G,K,I)$ is clique-Sperner (or determines that no such partition exists).
\end{lemma}

\begin{proof}
As usual, let us set $n = |V(G)|$ and $m = |E(G)|$.
The cases when $G$ has at most one edge are straightforward, so let us assume that $G$ has at least two edges. Suppose first that $(K,I)$ is any split partition of $G$ such that $(G,K,I)$ is clique-Sperner. Then $I$ is a maximal independent set of $G$, since otherwise there would exist a vertex $v\in K$ such that $N(v)\cap I = \emptyset$, which by the Sperner property implies that $K = \{v\}$ and consequently $G$ is edgeless, a contradiction.
Suppose now that there exists a split partition $(K',I')$ of $G$ different from $(K,I)$ such that $I'$ is a maximal independent set.
Since every vertex in $K'$ has a neighbor in $I'$, we infer that $K$ is not a proper subset of $K'$. Therefore, there exists a unique vertex $u\in I'$ such that $u\in K$. Moreover, $|K'\setminus K|\le 1$. Since $(G,K,I)$ is clique-Sperner, every vertex in $K$ must have a neighbor in $I$. Let $v\in N(u)\cap I$.
Since $u\in I'$ and $I'$ is independent, vertex $v$ must belong to $K'$.
I't follows that $K = (K'\setminus \{v\})\cup \{u\}$ and consequently $I = (I'\setminus\{u\})\cup\{v\}$. Since $I$ is independent, we have $N(v)\cap I' = \{u\}$. On the other hand, since $u\in K$ and $K$ is a clique, $u$ is adjacent to all vertices of $K'$. But now, $N(u)\cap I = \{v\}$.
Since every vertex of $K$ is adjacent to $v$ and $(G,K,I)$ is clique-Sperner, we infer that $K = \{u\}$. This implies that $E(G) = \{\{u,v\}\}$, a contradiction. The above contradiction shows that if $G$ has a split partition $(K,I)$ such that $(G,K,I)$ is clique-Sperner, then $(K,I)$ is a unique split partition of $G$ such that $I$ is a maximal independent set.

It remains to justify that there is an $\mathcal{O}(n^3)$ algorithm to compute a split partition $(K,I)$ of $G$ such that $(G,K,I)$ is clique-Sperner (or determine that no such partition exists). We may assume that $G$ has at least two edges. First, we compute in time $\mathcal{O}(n+m)$ some split partition $(K,I)$ of $G$~\cite{MR637832}. We verify in time $\mathcal{O}(n+m)$ whether $I$ is a maximal independent set of $G$; if it is not, then $G$ contains a vertex $v\in K$ such that $N(v)\subseteq K$ and we move one such vertex to $I$. After this operation, $I$ will be a maximal independent set.
Next, we test in time $\mathcal{O}(|K|^2|I|) = \mathcal{O}(n^3)$ if $(G,K,I)$ is clique-Sperner. This can be done in the stated time directly from the definition by computing all the neighborhoods in $I$ of vertices of $K$ and comparing each pair for inclusion. If $(G,K,I)$ is clique-Sperner, then we are done, and if this is not the case, then, as shown above, we infer that $G$ has no split partition with the desired property.
\end{proof}

\begin{remark}
There exist split graphs $G$ having a unique split partition $(K,I)$ such that $I$ is a maximal independent set but $(G,K,I)$ is not clique-Sperner.
\end{remark}

\begin{sloppypar}
To state the decomposition theorems, it will be useful to use the notion of matrix partitions of graphs; see Sections~\ref{sec:introduction} and~\ref{sec:preliminaries} for references and the definition.
Our theorems will make use of four specific matrices. For $(a,b)\in \{0,1\}^2$, let $M[a,b]$ be the following symmetric $5\times 5$ matrix with entries in $\{0,1,\ast\}$:
$$M[a,b] = \left(
        \begin{array}{ccccc}
          a & a & a & 1 & 0 \\
          a & a & a & \ast & 1 \\
          a & a & a & 0 & \ast \\
          1 & \ast & 0 & b & b \\
          0 & 1 & \ast & b & b \\
        \end{array}
      \right)\,.
$$
\end{sloppypar}

\begin{theorem}\label{thm:decomposition-clique-split-H-free}
Consider the graph $H$ depicted in Fig.~\ref{fig:small-graphs}.
Let $G$ be an $H$-free clique-Sperner split graph with at least two vertices and let $(K,I)$ be any split partition of $G$ such that the labeled split graph $(G,K,I)$ is clique-Sperner.
Then, there exists a partition of $K$ as $K = K^1\cup K^2$
and a partition of $I$ as $I = \{z\}\cup I^1\cup I^2$
such that $\{\{z\},I^1,I^2,K^1,K^2\}$ is an
$M[0,1]$-partition of $G$ (where the rows and the columns are indexed in order by $\{z\},I^1,I^2,K^1,K^2$) such that
the labeled split subgraphs $(G_1,K^1,I^1)$ and $(G_2,K^2,I^2)$ of $(G,K,I)$ induced by the sets $K^1\cup I^1$ and $K^2\cup I^2$, respectively (with inherited split partitions) are $H$-free and clique-Sperner.

Moreover, given an $H$-free split graph $G$, we can compute in time $\mathcal{O}(|V(G)|^3)$ a split partition $(K,I)$ as above, along with an $M[0,1]$-partition with the stated properties, or determine that $G$ is not
clique-Sperner.
\end{theorem}

\begin{proof}
As before, let $n = |V(G)|$ and $m = |E(G)|$.
If $I = \emptyset$, then $V(G) = K$ and since $(G,K,I)$ is clique-Sperner, we infer that $n = |K| = 1$, a contradiction to $n>1$.
So we have $I\neq\emptyset$. Consider the hypergraph $\mathcal{H} = (V_{\mathcal{H}},E_{\mathcal{H}})$ with vertex set $V_{\mathcal{H}} = I$, edge set $E_{\mathcal{H}} = \{e_v\mid v\in K\}$
where $e_v = N(v)\cap I$, that is, $e_v = \{u\in V_{\mathcal{H}}\mid uv\in E(G)\}$.

We claim that $\mathcal{H}$ is $1$-Sperner, that is, that every two distinct hyperedges $e$ and $f$ of $\mathcal{H}$ satisfy $\min\{|e\setminus f|,|f\setminus e|\} = 1$. Let $e,f\in E_{\mathcal{H}}$ be two distinct edges of $\mathcal{H}$. Then $e = e_v$ and $f = e_{v'}$ for some two vertices $v,v'\in K$. First we show that none of the edges $e$ and $f$ is contained in the other one. If $e\subseteq f$, then $N(v)\cap I= e_v\subseteq e_{v'} = N(v')\cap I$, which implies $v = v'$ since $(G,K,I)$ is clique-Sperner, hence $e = f$, a contradiction. By symmetry, $f\nsubseteq e$ and so $\min\{|e\setminus f|,|f\setminus e|\} \ge 1$, as claimed.

Next, suppose that $\min\{|e\setminus f|,|f\setminus e|\} \ge 2$.
Denoting by $x,y$ and $x',y'$ two pairs of distinct vertices of $\mathcal{H}$ such that $\{x,y\}\subseteq e_v\setminus e_{v'}$ and $\{x',y'\}\subseteq e_{v'}\setminus e_{v}$, we see that the subgraph of $G$ induced by $\{v,x,y,v',x',y'\}$ is isomorphic to $H$, contradicting the fact that $G$ is $H$-free. This shows that $\mathcal{H}$ is $1$-Sperner, as claimed.

Since $\mathcal{H} = (V_{\mathcal{H}},E_{\mathcal{H}})$ is a $1$-Sperner hypergraph with $V_{\mathcal{H}} = I \neq \emptyset$, Theorem~\ref{thm:decomposition} implies the existence of a vertex $z\in I$ such that $\mathcal{H}$ is $z$-decomposable. Let $z\in I$ denote such a vertex and let $\mathcal{H}$ be the corresponding gluing of $\mathcal{H}_1 = (V_1,E_1)$ and $\mathcal{H}_2 = (V_2,E_2)$.
That is, $V(\mathcal{H}) = V_1\cup V_2\cup\{z\}$ and $E_{\mathcal{H}} = \{\{z\}\cup e\mid e\in E_1\} \cup \{V_1\cup e\mid e\in E_2\}$.
Let us write $I^1 = V_1$ and $I^2 = V_2$.
Then $I$ is the disjoint union $I = I^1\cup I^2\cup\{z\}$ and letting $K^1 = N(z)$ and $K^2 = K\setminus K^1$,
the definition of
$\mathcal{H}$ and the fact that $\mathcal{H} = \mathcal{H}_1\odot \mathcal{H}_2$ implies that
\begin{equation}\label{eq-graph}
\begin{aligned}
E(G) &=& E(G[K])\cup \{zv\mid v\in K^1\} \cup\{uv\mid u\in I^1, v\in K^1, u\in e_v\in E_1\}\\
&&\cup\, \{uv\mid u\in I^1, v\in K^2\}\cup \{uv\mid u\in I^2, v\in K^2, u\in e_v\in E_2\}\,.
\end{aligned}
\end{equation}
See Fig.~\ref{fig:example3} for an example.

\begin{figure}[h!]
  \centering
   \includegraphics[width=\textwidth]{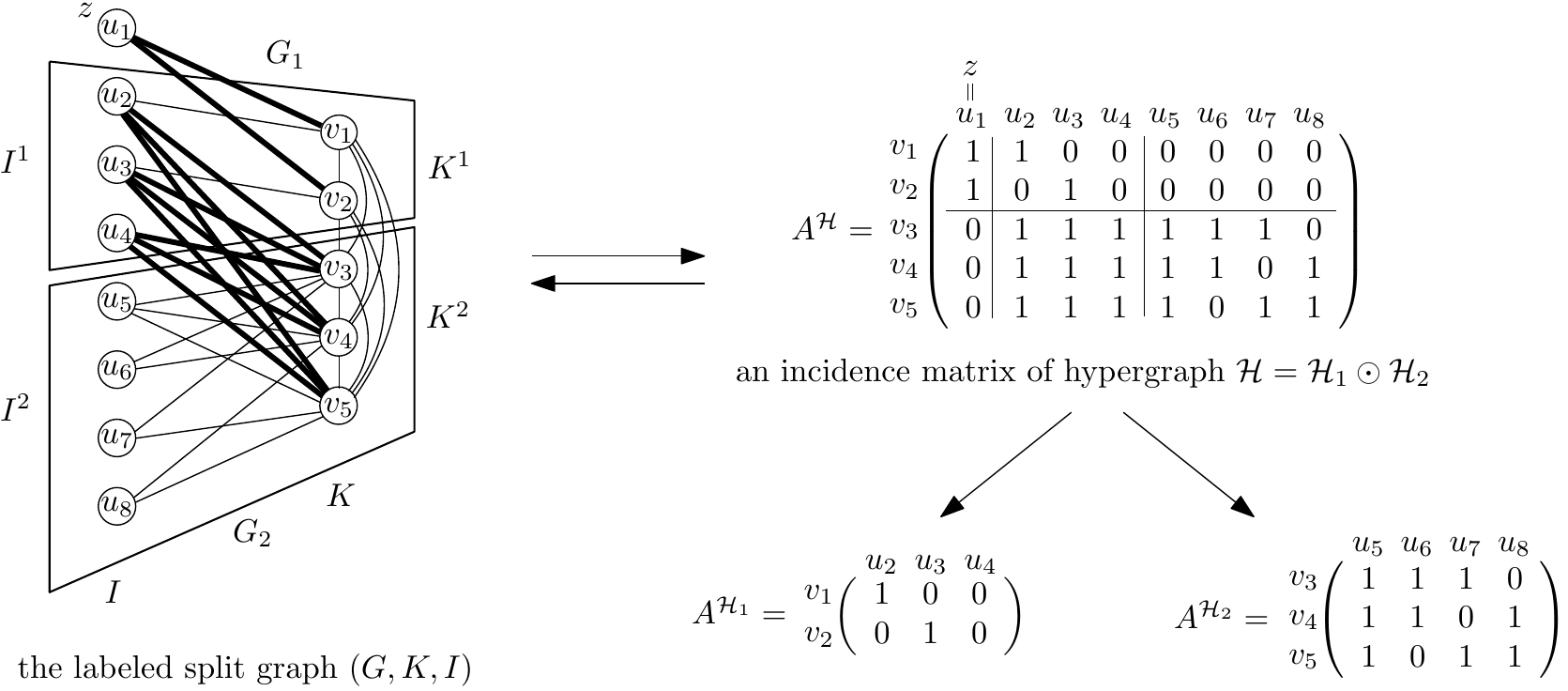}
\caption{Decomposing a given split graph $G$ using a $z$-decomposition of the derived $1$-Sperner hypergraph $\mathcal{H}$.} \label{fig:example3}
\end{figure}

\begin{sloppypar}
We claim that $\{\{z\},I^1,I^2,K^1,K^2\}$ is an
$M[0,1]$-partition of $G$ with the desired property. Equation~\eqref{eq-graph} and the fact that $K$ is a clique and $I$ an independent set imply that $\{\{z\},I^1,I^2,K^1,K^2\}$ is indeed an
$M[0,1]$-partition of $G$. For $i \in \{1,2\}$, let us denote by
$(G_i,K^i,I^i)$ the labeled split subgraph of $(G,K,I)$ induced by the set $K^i\cup I^i$. Then, $(G_i,K^i,I^i)$ is isomorphic to the edge-clique split graph of $\mathcal{H}_i$. By Observation~\ref{obs:constituents}, since $\mathcal{H}$ is $1$-Sperner, so is $\mathcal{H}_i$. Hence, by~\Cref{prop:translation-to-graphs}, the edge-clique  split graph of $\mathcal{H}_i$, that is, $(G_i,K^i,I^i)$, is clique-Sperner. Of course, it is also $H$-free.
\end{sloppypar}

It remains to justify that the above decomposition can be computed in time $\mathcal{O}(n^3)$. Using Lemma~\ref{lem:split}, we can compute in time $\mathcal{O}(n^3)$
a split partition $(K,I)$ of $G$ such that the labeled graph $(G,K,I)$ is clique-Sperner (or determine that $G$ is not clique-Sperner). It is not difficult to see that $G$ has an $M[0,1]$-partition with the desired properties if and only if there exists a vertex $z$ in $I$ such that if we denote by $N_2(z)$ the set of vertices at distance two in $G$ from $z$ and set $I^1 = I\cap N_2(z)$, $K^2 = K\cap N_2(z)$, then $G$ contains all possible edges between vertices of $I^1$ and vertices of $K^2$. If $z\in I$ is such a vertex, then we set $I^2 = I\setminus(\{z\}\cup I^1)$ and $K^1 = K\setminus K^2$, and
$\{\{z\},I^1,I^2,K^1,K^2\}$ forms an $M[0,1]$-partition of $G$ with the desired properties. The above test can be performed in time $\mathcal{O}(n+m) = \mathcal{O}(n^2)$ per vertex, resulting in a total running time of $\mathcal{O}(n^3)$, as claimed.
\end{proof}

The next theorem states the result analogous to~\Cref{thm:decomposition-clique-split-H-free}
but for $\overline{H}$-free independent-Sperner split graphs.

\begin{theorem}\label{thm:decomposition-independent-split-H-bar-free}
Let $G$ be an $\overline{H}$-free independent-Sperner split graph with at least two vertices and let $(K,I)$ be a split partition of $G$ such that the labeled split graph $(G,K,I)$ is independent-Sperner.
Then, there exists a partition of $K$ as $K = \{z\}\cup K^1\cup K^2$
and a partition of $I$ as $I = I^1\cup I^2$
such that $\{\{z\},K^1,K^2,I^1,I^2\}$ is an
$M[1,0]$-partition of $G$ (where the rows and the columns are indexed in order by $\{z\},K^1,K^2,I^1,I^2$) such that
the labeled split subgraphs $(G_1,K^1,I^1)$ and $(G_2,K^2,I^2)$ of $(G,K,I)$ induced by the sets $K^1\cup I^1$ and $K^2\cup I^2$, respectively (with inherited split partitions) are $\overline{H}$-free and independent-Sperner.

Moreover, given an $\overline{H}$-free split graph $G$, we can compute in time $\mathcal{O}(|V(G)|^3)$ a split partition $(K,I)$ as above, along with an $M[1,0]$-partition with the stated properties, or determine that $G$ is not independent-Sperner.
\end{theorem}

\begin{proof}
Apply Theorem~\ref{thm:decomposition-clique-split-H-free} to the complement of $G$.
\end{proof}

The following two theorems can be shown analogously, for proofs see Appendix.

\begin{theorem}\label{thm:decomposition-2P3-free-right-Sperner-bigraphs}
Let $G$ be a $2P_3$-free right-Sperner bigraph with at least two vertices and let $(A,B)$ be a bipartition of $G$ such that the labeled bigraph $(G,A,B)$ is right-Sperner. Then, there exists a partition of $A$ as $A = \{z\}\cup A^1\cup A^2$ and a partition of $B$ as $B = B^1\cup B^2$
such that $\{\{z\},A^1,A^2,B^1,B^2\}$ is an
$M[0,0]$-partition of $G$ (where the rows and the columns are indexed in order by $\{z\},A^1,A^2,B^1,B^2$) such that
the labeled bipartite subgraphs $(G_1,A^1,B^1)$ and $(G_2,A^2,B^2)$ of $(G,A,B)$ induced by the sets $A^1\cup B^1$ and $A^2\cup B^2$, respectively (with inherited bipartitions), are $2P_3$-free and right-Sperner.

Moreover, given an $2P_3$-free bigraph $G$, we can compute in time $\mathcal{O}(|V(G)|^3)$ a bipartition $(A,B)$ as above, along with an $M[0,0]$-partition with the stated properties, or determine that $G$ is not right-Sperner.
\end{theorem}

\begin{sloppypar}
\begin{theorem}\label{thm:decomposition-co-2P3-free-right-Sperner-cobigraphs}
Let $G$ be a $\overline{2P_3}$-free cobipartite graph with at least two vertices such that $\overline{G}$ is right-Sperner
and let $(A,B)$ be a bipartition of $\overline{G}$ such that the labeled bigraph $(\overline{G},A,B)$ is right-Sperner.
Then, there exists a partition of $A$ as $A = \{z\}\cup A^1\cup A^2$ and a partition of $B$ as $B = B^1\cup B^2$
such that $\{\{z\},A^1,A^2,B^1,B^2\}$ is an
$M[1,1]$-partition of $G$ (where the rows and the columns are indexed in order by $\{z\},A^1,A^2,B^1,B^2$) such that
the subgraphs of $G$ induced by the sets $A^1\cup B^1$ and $A^2\cup B^2$, respectively, are $\overline{2P_3}$-free and their complements are right-Sperner.

Moreover, given a $\overline{2P_3}$-free cobipartite graph $G$, we can compute in time $\mathcal{O}(|V(G)|^3)$ a bipartition $(A,B)$of $\overline{G}$ as above, along with an $M[1,1]$-partition of $G$ with the stated properties, or determine that $\overline{G}$ is not right-Sperner.
\end{theorem}
\end{sloppypar}

\section{Implications for clique-width}\label{sec:clique-width}

Building on previous work of Lozin and Volz~\cite{MR2414868}, Dabrowski and Paulusma~\cite{MR3442572} obtained a complete classification of bigraphs $F$ such that the class of $F$-free bigraphs is of bounded clique-width. Brandst\"adt et al.~\cite{MR3515312} obtained similar results for the classes of $F$-free split graphs, for a split graph $F$. In particular, results from~\cite{MR3515312,MR2414868} imply the following theorem.

\begin{theorem}\label{thm:unbounded}
The classes of $H$-free split graphs, of $\overline{H}$-free split graphs,
and of $2P_3$-free bigraphs are of unbounded clique-width.
\end{theorem}

In Sections~\ref{sec:to-graphs} and~\ref{sec:structure-graphs}
we developed, using a unifying approach based on $1$-Sperner hypergraphs,
decomposition theorems for classes of graphs considered by \Cref{thm:unbounded}, under an additional constraint on the structure of neighborhoods. Contrary to the result of \Cref{thm:unbounded},
we now show that these additional constraints and the consequent
structural results imply that the clique-width of graphs in the corresponding classes is bounded by 5.

We first consider the case of $H$-free clique-Sperner split graphs.

\begin{theorem}\label{thm:cwd-clique-split}
The clique-width of every $H$-free clique-Sperner split graph is at most $5$. Moreover, given an $H$-free clique-Sperner split graph $G$, we can compute in time $\mathcal{O}(|V(G)|^4)$ a $5$-expression $\tau$ of $G$
such that $|\tau|= O(|V(G)|)$.
\end{theorem}

\begin{proof}
Let $G$ be an $H$-free clique-Sperner split graph.
As usual, we set $n = |V(G)|$.
Using Lemma~\ref{lem:split}, we can compute in time $\mathcal{O}(n^3)$
a split partition $(K,I)$ of $G$ such that the labeled graph $(G,K,I)$ is clique-Sperner.

We show by induction on $n$ that $G$ admits a $5$-expression
of length at most $60n$ such that the label of every vertex in $I$ is from $\{1,2,3\}$  and the label of every vertex in $K$ is from $\{4,5\}$. If $I = \emptyset$, then $V(G) = K$ and since $(G,K,I)$ is clique-Sperner, we infer that $n = |K| = 1$ and the statement holds in this case. So let $I\neq\emptyset$. By Theorem~\ref{thm:decomposition-clique-split-H-free},
we can compute in time $\mathcal{O}(n^3)$ an $M[0,1]$-partition $\{\{z\},I^1,I^2,K^1,K^2\}$ of $G$ (where the rows and the columns are indexed in order by $\{z\},I^1,I^2,K^1,K^2$) such that
the labeled split subgraphs $(G_1,K^1,I^1)$ and $(G_2,K^2,I^2)$ of $(G,K,I)$ induced by the sets $K^1\cup I^1$ and $K^2\cup I^2$, respectively (with inherited split partitions) are $H$-free and clique-Sperner.
For $i\in \{1,2\}$, let $n_i = |V(G_i)|$.
By the induction hypothesis, $G_i$ admits a $5$-expression $\tau_i$ of length at most $60n_i$ such that the label of every vertex in $V_i$ is from $\{1,2,3\}$  and
the label of every vertex in $K_i$ is from $\{4,5\}$.
Let
\begin{eqnarray*}
\tau_1' &=& \rho_{1\to 2}(\rho_{3\to 2}(\rho_{5\to 4}(\tau_1)))\,, \\
\tau_2' &=& \rho_{1\to 3}(\rho_{2\to 3}(\rho_{4\to 5}(\tau_2)))\,, \\
\tau &=& \eta_{1,4}(\eta_{2,5}(\eta_{4,5}(1(z)\oplus \tau_1'\oplus \tau_2')))\,.
\end{eqnarray*}
The fact that $\{\{z\},I^1,I^2,K^1,K^2\}$ is an $M[0,1]$-partition of $G$
implies that $\tau$ is a $5$-expression of $G$ such that
the label of every vertex in $I$ is from $\{1,2,3\}$  and
the label of every vertex in $K$ is from $\{4,5\}$.
Moreover, we have $|\tau| = |\tau_1|+|\tau_2|+60\le 60n$.
This completes the proof of the inductive step.

The inductive proof can be turned into a polynomial-time algorithm.
Every time we remove one vertex (denoted by $z$ in the above argument) and call the algorithm recursively on two smaller vertex-disjoint subgraphs $(G_1,K^1,I^1)$ and $(G_2,K^2,I^2)$.
If we denote the worst-case computing time on an $n$-vertex graph by $T(n)$,
then we see that $T(n)$ satisfies the following recursion:
$$T(n_1+n_2+1)\le T(n_1)+T(n_2)+\mathcal{O}((n_1+n_2+1)^3)\,.$$
It is not difficult to see that $T(n) = \mathcal{O}(n^4)$.
\end{proof}

The following two theorems can be proved analogously.

\begin{sloppypar}
\begin{theorem}\label{thm:cwd-independent-split}
The clique-width of every $\overline{H}$-free \hbox{independent-Sperner} split graph is at most $5$. Moreover, given an $\overline{H}$-free \hbox{independent-Sperner} split graph $G$, we can compute a $5$-expression of $G$ in time $\mathcal{O}(|V(G)|^4)$.
\end{theorem}
\end{sloppypar}

\begin{theorem}\label{thm:cwd-bigraph}
The clique-width of every $2P_3$-free right-Sperner bigraph is at most $5$. Moreover, given a $2P_3$-free right-Sperner bigraph $G$, we can compute a $5$-expression of $G$ in time $\mathcal{O}(|V(G)|^4)$.
\end{theorem}

\citet{MR1743732} showed that for any graph $G$, the clique-width of $G$ and its complement $\overline{G}$ are related by the following inequality
\begin{equation}\label{eq:complement}
\cw(\overline{G})\le 2\cw(G)\,.
\end{equation}
Using \Cref{thm:unbounded},
inequality~\eqref{eq:complement} implies that the class of $\overline{2P_3}$-free cobipartite graphs is of unbounded clique-width.
The same inequality combined with
\Cref{thm:cwd-bigraph} implies that the clique-width of every $\overline{2P_3}$-free cobipartite graph $G$ such that $\overline{G}$ is a right-Sperner bigraph is at most $10$.
In fact, the same approach as the one used to prove Theorems~\ref{thm:cwd-clique-split}, \ref{thm:cwd-independent-split}, and~\ref{thm:cwd-bigraph} can be naturally adapted also to the cobipartite setting, showing the following.

\begin{sloppypar}
\begin{theorem}\label{thm:cwd-cobigraph}
The clique-width of every $\overline{2P_3}$-free cobipartite graph $G$ whose complement is \hbox{right-Sperner} is at most $5$. Moreover, given such a graph $G$, we can compute a $5$-expression of $G$ in time $\mathcal{O}(|V(G)|^4)$.
\end{theorem}
\end{sloppypar}

\section{Consequences for domination problems}\label{sec:domination}

Theorem~\ref{thm:cwd-clique-split} is interesting not only in contrast with Theorem~\ref{thm:unbounded} but also has algorithmic consequences. As a corollary of Theorem~\ref{thm:cwd-clique-split}, we obtain new polynomially solvable cases of three basic variants of the dominating set problem,
the {\sc Dominating Set}, {\sc Total Dominating Set} and {\sc Connected Dominating Set} problems. All these problems can be formulated in Monadic Secod Order Logic with quantifiers over vertices and vertex subsets. Therefore, by the meta-theorem of Courcelle et al.~\cite{MR1739644}, the problem can be solved in time $\mathcal{O}(|\tau|)$ on graphs given by a $k$-expression $\tau$ for any constant value of $k$. In particular, Theorem~\ref{thm:cwd-clique-split} implies that
the {\sc Dominating Set}, {\sc Total Dominating Set}, and {\sc Connected Dominating Set} problems can be solved in time
$\mathcal{O}(|V(G)|^4)$ for $H$-free clique-Sperner split graphs.
Furthermore, a cubic-time algorithm can be obtained by using a result of Oum~\cite{MR2479181} stating that there exists a function $f$ such that for every constant $k$, given a graph $G$ of clique-width at most $k$, we can compute in time $\mathcal{O}(|V(G)|^3)$ an $(8^k-1)$-expression of $G$.

\begin{sloppypar}
\begin{corollary}\label{lem:cwd}
The {\sc Dominating Set}, {\sc Total Dominating Set} and {\sc Connected Dominating Set} problems can be solved in time
$\mathcal{O}(|V(G)|^3)$ for $H$-free clique-Sperner split graphs.
\end{corollary}
\end{sloppypar}

The next lemma shows that the three domination problems are closely interrelated in the class of split graphs. This fact that will be instrumental for proving \Cref{cor:dom-sets} and \Cref{thm:dom-sets}.

\begin{lemma}\label{lem:reductions}
The {\sc Dominating Set}, {\sc Total Dominating Set}, and {\sc Connected Dominating Set} problems are linear-time equivalent in any class of connected split graphs.
\end{lemma}

\begin{proof}
Throughout this proof, let $\mathcal{G}$ be a class of connected split graphs and let $G = (V,E)$ be a split graph from $\mathcal{G}$, with a split partition $(K,I)$. We may assume without loss of generality that $G$ has at least two vertices.

Suppose first that the {\sc Dominating Set} problem can be solved in polynomial time for graphs in $\mathcal{G}$. Since every vertex in $I$ has a neighbor in $K$, graph $G$ has a minimum dominating set $D^*$ such that $D^*\subseteq K$ (every dominating set $D$ of $G$ with $v\in I\cap D$ can be transformed into a dominating set $D'$ such that $|D'|\le |D|$ by replacing $v$ with an arbitrary neighbor of it). Such a set $D^*$ can be computed in polynomial time. Consider the {\sc Total Dominating Set} problem. Since every total dominating set is a dominating set, we have $\gamma(G)\le \gamma_t(G)$. If $\gamma(G) = 1$, then $\gamma_t(G) = 2$ and any set of the form $\{u,v\}$ where $\{u\}$ is a dominating set in $G$ and $uv\in E(G)$ is a minimum total dominating set in $G$. If $\gamma(G) \ge 2$, then, note that since $|D^*|\ge 2$ and $D^*$ is a clique, $D^*$ is also a total dominating set in $G$. Hence
$\gamma_t(G)\le |D^*|\le \gamma(G)\le \gamma_t(G)$, hence in this case $D^*$ is a minimum total dominating set. The argument for the {\sc Connected Dominating Set} problem is even simpler. Since every connected dominating set is a dominating set, we have $\gamma(G)\le \gamma_c(G)$. Moreover, since $D^*$ is a connected dominating set, equality holds and $D^*$ is a minimum connected dominating set.

Suppose that the {\sc Total Dominating Set} problem can be solved in polynomial time for graphs in $\mathcal{G}$. We will show that the {\sc Dominating Set} problem is also polynomially solvable for graphs in $\mathcal{G}$. We may assume that $\gamma(G)\ge 2$ (the case $\gamma(G) = 1$ can be identified in polynomial time). Let $D^*$ be a minimum total dominating set in $G$. Clearly, $D^*$ is a dominating set. Moreover, $D^*$ is also a minimum dominating set: if $G$ had a dominating set $D'\subseteq K$ such that $|D'|<|D^*|$, then, since $|D'|\ge \gamma(G)\ge 2$ and $D'$ is a clique, $D'$ would be a total dominating set of $G$ smaller than $D^*$, a contradiction. The argument is similar if the {\sc Connected Dominating Set} problem is polynomially solvable for graphs in $\mathcal{G}$. Assuming that $\gamma(G)\ge 2$, every minimum connected dominating set $D^*$ in $G$ is also a minimum dominating set: if $G$ had a dominating set $D'\subseteq K$ such that $|D'|<|D^*|$, then, since $D'$ would be a connected dominating set of $G$ smaller than $D^*$, a contradiction.
\end{proof}

The following result is a consequence of a more general result~\cite[Corollary 6.7]{CM-ISAIM2014}.

\begin{theorem}[\citet{CM-ISAIM2014}]\label{thm:CM-dom-sets}
The {\sc Connected Dominating Set} problem is solvable in time $\mathcal{O}(|V(G)|^3)$ in the class of $\overline H$-free split graphs.
\end{theorem}

Lemma~\ref{lem:reductions} and Theorem~\ref{thm:CM-dom-sets} imply the following.

\begin{sloppypar}
\begin{corollary}\label{cor:dom-sets}
The {\sc Dominating Set} and {\sc Total Dominating Set} problems are solvable in time $\mathcal{O}(|V(G)|^3)$ in the class of $\overline H$-free split graphs.
\end{corollary}
\end{sloppypar}

We now complement the results of Theorem~\ref{thm:CM-dom-sets} and Corollary~\ref{cor:dom-sets} as follows.

\begin{theorem}\label{thm:dom-sets}
The {\sc Dominating Set}, {\sc Total Dominating Set}, and {\sc Connected Dominating Set} problems are solvable in time $\mathcal{O}(|V(G)|^3)$
in the class of $H$-free split graphs.
\end{theorem}

\begin{proof}
Let $G = (V,E)$ be an $H$-free split graph, with a split partition $(K,I)$.
Let $n = |V(G)|$ and $m = |E(G)|$. We may assume without loss of generality that $G$ is connected and has at least two vertices. By Lemma~\ref{lem:reductions}, it suffices to show that {\sc Dominating Set} is solvable in time $\mathcal{O}(n^3)$ for $G$. Since $G$ is connected and has at least two vertices, every vertex in $I$ has a neighbor in $K$. This implies that $G$ has a minimum dominating set $D$ such that $D\subseteq K$.

Let us now show that we may assume that $G$ is clique-Sperner. Suppose that this is not the case. Then, there exists a pair $u,v\in K$ of distinct vertices such that $N(u)\cap I \subseteq N(v)\cap  I$. We claim that in this case the problem can be reduced to the graph $G-u$, in the sense that
$\gamma(G)= \gamma(G-u)$. Note that since $G$ is connected, $G-u$ is also connected.

To see the claimed equality, note first that every minimum dominating set of $G-u$ contained in $K\setminus\{u\}$ is a minimum dominating set of $G$.
Indeed, if $D$ is a minimum dominating set of $G-u$, then $u$ has a neighbor in $D$ since $D\subseteq K$ and therefore $D$ is a dominating set of $G$.
This shows $\gamma(G)\le \gamma(G-u)$. Secondly, if $D$ is a minimum dominating set of $G$, then either $u\not\in D$, in which case $D$ is a dominating set of $G-u$, or $u\in D$, in which case $(D\setminus \{u\})\cup\{v\}$ is a dominating set in $G-u$. This shows that $\gamma(G-u)\le \gamma(G)$ and establishes the correctness of the reduction.

Therefore, we may assume that $G$ is an $H$-free clique-Sperner split graph
and we can invoke \Cref{lem:cwd} to conclude that the dominating set problem is solvable in time $\mathcal{O}(n^3)$ in the class of $H$-free split graphs.

The time complexity can be justified as follows. The reduction to the case of connected graphs can be performed in time $\mathcal{O}(n+m)$, by considering each connected component separately. A split partition $(K,I)$ can be computed in time $\mathcal{O}(n+m)$~\cite{MR637832}. Then, in time $\mathcal{O}(n^3)$, graph $G$ can be reduced to an equivalent
clique-Sperner induced subgraph $G'$ of $G$.
This can be done as follows. First we compute the equivalence classes of the relation on $K$ defined by $u\sim v$ if and only if $N(u)\cap I = N(v)\cap I$. Then we delete vertices from $K$ so that only one representative of each class remains. Call the resulting set $K^1$. Then we compute the set $K^2$ of vertices $u\in K^1$ such that there exists a vertex $v\in K^1$ such that $N(u)\cap I \subset N(v)\cap I$. The graph $G'$ is then $G[(K^1\setminus K^2 )\cup I]$. Since $G$ is $H$-free, so is $G'$. Moreover, $G'$ is clique-Sperner. By \Cref{lem:cwd}, a minimum dominating set in $G'$ can be computed in time $\mathcal{O}(n^3)$.
\end{proof}

We conclude by observing that Theorem~\ref{thm:CM-dom-sets}, \Cref{cor:dom-sets}, and \Cref{thm:dom-sets} are sharp, in the sense that all three domination problems become \NP-hard if any of the assumptions is dropped.
Indeed, all the three problems are \NP-hard in the class of split graphs~\cite{LP83,MR761623,MR754426} (and even hard to approximate, see Introduction), in the class of line graphs~\cite{MR579424,McRae,MR2874128}, and consequently in the class of claw-free graphs. Since every claw-free graph is $H$-free, the problems are also \NP-hard in the class of $H$-free graphs. Similarly, since the
problems are known to be \NP-hard in the class of bipartite graphs~\cite{D81,PLH83}, and every bipartite graph is triangle-free and thus $\overline{H}$-free, the problems are also \NP-hard in the class of $\overline{H}$-free graphs.

\subsection*{Acknowledgements}

\begin{sloppypar}
The second author was partially funded by the Russian Academic Excellence Project `5-100'. The work for this paper was done in the framework of bilateral projects between Slovenia and the USA, partially financed by the Slovenian Research Agency (BI-US/$14$--$15$--$050$, BI-US/$16$--$17$--$030$, and BI-US/$18$--$19$--$029$). The work of the third author is supported in part by the Slovenian Research Agency (I$0$-$0035$, research program P$1$-$0285$, research projects N$1$-$0032$, J$1$-$6720$, and J$1$-$7051$).
\end{sloppypar}

\def\ocirc#1{\ifmmode\setbox0=\hbox{$#1$}\dimen0=\ht0 \advance\dimen0
  by1pt\rlap{\hbox to\wd0{\hss\raise\dimen0
  \hbox{\hskip.2em$\scriptscriptstyle\circ$}\hss}}#1\else {\accent"17 #1}\fi}

\newpage
\section*{Appendix: Proofs of Theorems~\ref{thm:decomposition-2P3-free-right-Sperner-bigraphs}
and~\ref{thm:decomposition-co-2P3-free-right-Sperner-cobigraphs}}

\begin{sloppypar}
\begin{thm-bigraph}[restated]
Let $G$ be a $2P_3$-free right-Sperner bigraph with at least two vertices and let $(A,B)$ be a bipartition of $G$ such that the labeled bigraph $(G,A,B)$ is right-Sperner. Then, there exists a partition of $A$ as $A = \{z\}\cup A^1\cup A^2$ and a partition of $B$ as $B = B^1\cup B^2$
such that $\{\{z\},A^1,A^2,B^1,B^2\}$ is an
$M[0,0]$-partition of $G$ (where the rows and the columns are indexed in order by $\{z\},A^1,A^2,B^1,B^2$) such that
the labeled bipartite subgraphs $(G_1,A^1,B^1)$ and $(G_2,A^2,B^2)$ of $(G,A,B)$ induced by the sets $A^1\cup B^1$ and $A^2\cup B^2$, respectively (with inherited bipartitions), are $2P_3$-free and right-Sperner.

Moreover, given an $2P_3$-free bigraph $G$, we can compute in time $\mathcal{O}(|V(G)|^3)$ a bipartition $(A,B)$ as above, along with an $M[0,0]$-partition with the stated properties, or determine that $G$ is not right-Sperner.
\end{thm-bigraph}
\end{sloppypar}

\begin{proof}
For brevity, let us set $n = |V(G)|$ and $m = |E(G)|$.
If $A = \emptyset$, then $V(G) = B$ and since $(G,A,B)$ is right-Sperner, we infer that $n = |B| = 1$, a contradiction to $n>1$.
So we have $A \neq \emptyset$. Consider the hypergraph $\mathcal{H} = (V_{\mathcal{H}},E_{\mathcal{H}})$ with vertex set $V_{\mathcal{H}} = A$, edge set $E_{\mathcal{H}} = \{e_v\mid v\in B\}$
where $e_v = N(v)$, that is, $e_v = \{u\in V_{\mathcal{H}}\mid uv\in E(G)\}$.

We claim that $\mathcal{H}$ is $1$-Sperner, that is, that every two distinct hyperedges $e$ and $f$ of $\mathcal{H}$ satisfy $\min\{|e\setminus f|,|f\setminus e|\} = 1$. Let $e,f\in E_{\mathcal{H}}$ be two distinct edges of $\mathcal{H}$. Then $e = e_v$ and $f = e_{v'}$ for some two vertices $v,v'\in K$. First we show that none of the edges $e$ and $f$ is contained in the other one. If $e\subseteq f$, then $N(v)= e_v\subseteq e_{v'} = N(v')$, which implies $v = v'$ since $(G,A,B)$ is right-Sperner, hence $e = f$, a contradiction. By symmetry, $f\nsubseteq e$ and so $\min\{|e\setminus f|,|f\setminus e|\} \ge 1$, as claimed.

Next, suppose that $\min\{|e\setminus f|,|f\setminus e|\} \ge 2$.
Denoting by $x,y$ and $x',y'$ two pairs of distinct vertices of $\mathcal{H}$ such that $\{x,y\}\subseteq e_v\setminus e_{v'}$ and $\{x',y'\}\subseteq e_{v'}\setminus e_{v}$, we see that the subgraph of $G$ induced by $\{v,x,y,v',x',y'\}$ is isomorphic to $2P_3$, contradicting the fact that $G$ is $2P_3$-free. This shows that $\mathcal{H}$ is $1$-Sperner, as claimed.

Since $\mathcal{H} = (V_{\mathcal{H}},E_{\mathcal{H}})$ is a $1$-Sperner hypergraph with $V_{\mathcal{H}} = A \neq \emptyset$, Theorem~\ref{thm:decomposition} implies the existence of a vertex $z\in A$ such that $\mathcal{H}$ is $z$-decomposable. Let $z\in A$ denote such a vertex and let $\mathcal{H}$ be the corresponding gluing of $\mathcal{H}_1 = (V_1,E_1)$ and $\mathcal{H}_2 = (V_2,E_2)$.
That is, $V(\mathcal{H}) = V_1\cup V_2\cup\{z\}$ and $E_{\mathcal{H}} = \{\{z\}\cup e\mid e\in E_1\} \cup \{V_1\cup e\mid e\in E_2\}$.
Let us write $A^1 = V_1$ and $A^2 = V_2$.
Then $A$ is the disjoint union $A = A^1\cup A^2\cup\{z\}$ and letting $B^1 = N(z)$ and $B^2 = B\setminus B^1$,
the definition of
$\mathcal{H}$ and the fact that $\mathcal{H} = \mathcal{H}_1\odot \mathcal{H}_2$ implies that
\begin{equation}\label{eq-graph-33}
\begin{aligned}
E(G) &=& \{zv\mid v\in B^1\} \cup\{uv\mid u\in V_1, v\in B^1, u\in e_v\in E_1\}~~~~~~~~~~~~~~~\\
&&\cup\, \{uv\mid u\in V_1, v\in B^2\}\cup \{uv\mid u\in V_2, v\in B^2, u\in e_v\in E_2\}\,.
\end{aligned}
\end{equation}

\begin{sloppypar}
We claim that $\{\{z\},A^1,A^2,B^1,B^2\}$ is an
$M[0,0]$-partition of $G$ with the desired property. Equation~\eqref{eq-graph-33} and the fact that $A$ and $B$ are independent sets imply that $\{\{z\},A^1,A^2,B^1,B^2\}$  is indeed an
$M[0,0]$-partition of $G$.
For $i \in \{1,2\}$, let us denote by
$(G_i,A^i,I^i)$ the labeled bipartite subgraph if $(G,A,B)$ induced by the sets $A^i\cup B^1$ and $A^2\cup B^2$, respectively.
Then, $(G_i,A^i,B^i)$ is isomorphic to the bigraph of $\mathcal{H}_i$. By Observation~\ref{obs:constituents}, since $\mathcal{H}$ is $1$-Sperner, so is $\mathcal{H}_i$. Hence, by~\Cref{prop:translation-to-graphs}, the bigraph of $\mathcal{H}_i$, that is, $(G_i,A_i,B_i)$, is right-Sperner.
Of course, it is also $2P_3$-free.
\end{sloppypar}

It remains to justify that the above decomposition can be computed in time $\mathcal{O}(n^3)$. First, note that since $G$ is $2P_3$-free,
$G$ has at most one component with more than two vertices.
Thus, in order to determine a bipartition $(A,B)$ of $G$ such that the labeled bigraph $(G,A,B)$ is right-Sperner, we can place, without loss of generality, all isolated vertices of $G$ into $A$, as well as one vertex from each component of $G$ isomorphic to $K_2$. If $G$ has a component, say $G'$, with more than two vertices, then we test for each of the two bipartitions $(A',B')$ of $G'$ whether the labeled bigraph $(G',A',B')$ is right-Sperner. This can be done in time $\mathcal{O}(|B'|^2|A'|) = \mathcal{O}(n^3)$  directly from the definition by computing all the neighborhoods of vertices of $B'$ and comparing each pair for inclusion.
If none of them is right-Sperner, then we infer that $G$ is not right-Sperner. Otherwise, we have obtained a bipartition $(A,B)$ of $G$ such that $(G,A,B)$ is right-Sperner. It is now not difficult to see that $G$ has an $M[0,0]$-partition with the desired properties if and only if there exists a vertex $z$ in $A$ such that if we denote by $B^1$ the neighborhood in $G$ of $z$, by $A^1$ the set of all vertices at distance two in $G$ from $z$,
and by $B^2$ the set $B\setminus B^1$, then $G$ contains all possible edges between vertices of $A^1$ and vertices of $B^2$. If $z\in A$ is such a vertex, then we set $A^2 = A\setminus(\{z\}\cup A^1)$ and
$\{\{z\},A^1,A^2,B^1,B^2\}$ forms an $M[0,0]$-partition of $G$ with the desired properties. The above test can be performed in time $\mathcal{O}(n+m) = \mathcal{O}(n^2)$ per vertex, resulting in a total running time of $\mathcal{O}(n^3)$, as claimed.
\end{proof}

\begin{sloppypar}
\begin{thm-cobigraph}[restated]
Let $G$ be a $\overline{2P_3}$-free cobipartite graph with at least two vertices such that $\overline{G}$ is right-Sperner
and let $(A,B)$ be a bipartition of $\overline{G}$ such that the labeled bigraph $(\overline{G},A,B)$ is right-Sperner.
Then, there exists a partition of $A$ as $A = \{z\}\cup A^1\cup A^2$ and a partition of $B$ as $B = B^1\cup B^2$
such that $\{\{z\},A^1,A^2,B^1,B^2\}$ is an
$M[1,1]$-partition of $G$ (where the rows and the columns are indexed in order by $\{z\},A^1,A^2,B^1,B^2$) such that
the subgraphs of $G$ induced by the sets $A^1\cup B^1$ and $A^2\cup B^2$, respectively, are $\overline{2P_3}$-free and their complements are right-Sperner.

Moreover, given a $\overline{2P_3}$-free bigraph, a bipartition $(A,B)$ as above, along with an $M[1,1]$-partition with the stated properties, can be computed in time $\mathcal{O}(|V(G)|^3)$.
\end{thm-cobigraph}
\end{sloppypar}

\begin{proof}
Apply Theorem~\ref{thm:decomposition-2P3-free-right-Sperner-bigraphs} to the complement of $G$.
\end{proof}

\end{document}